\documentclass[10pt]{amsart}
\usepackage[utf8]{inputenc}

\usepackage{amsmath}               
\usepackage{amsfonts}              
\usepackage{amsthm}                
\usepackage[utf8]{inputenc}
\usepackage{leftidx}
\usepackage{tikz-cd}
\usepackage{mathtools}
\usepackage[foot]{amsaddr}

\newtheorem{tw}{Theorem}[section]
\newtheorem{lem}[tw]{Lemma}
\newtheorem{defi}[tw]{Definition}
\newtheorem{prop}[tw]{Proposition}
\newtheorem{rem}[tw]{Remark}
\newtheorem{cor}[tw]{Corollary}
\newtheorem{egg}[tw]{Example}

\title{Covering spaces of symplectic toric orbifolds}
\author{Paweł Raźny$^1$}
\address{$^1$Institute of Mathematics \\
	Faculty of Mathematics and Computer Science \\
	Jagiellonian University in Cracow
	}
\author{Nikolay Sheshko$^2$}
\address{$^2$Department of Mathematics \\
University of Illinois at Urbana-Champaign}
	\keywords{Symplectic toric orbifolds} \subjclass[2010]{57S12}
\date{}
\begin{document}

\maketitle
\begin{abstract} In this article we study covering spaces of symplectic toric orbifolds and symplectic toric orbifold bundles. In particular, we show that all symplectic toric orbifold coverings are quotients of some symplectic toric orbifold by a finite subgroup of a torus. We then give a general description of the labeled polytope of a toric orbifold bundle in terms of the polytopes of the fiber and the base. Finally, we apply our findings to study the number of toric structures on products of labeled projective spaces.
\end{abstract}
\section{Introduction}
The initial point of this article is the observation that in sharp contrast to the manifold case (where it is well known that compact symplectic toric manifolds are simply connected) there exist non-trivial orbifold coverings of toric symplectic orbifolds. A simple precedence for this is given by taking the quotient $\tilde{X}\slash G$ of a symplectic toric orbifold $\tilde{X}$ by a finite subgroup $G\subset\mathbb{T}^n$. In fact, we prove that all connected coverings of a connected symplectic toric orbifold are of this form. This naturally leads us to the study of toric orbifold bundles and the labeled polytope of its total space. In particular, we provide a description of the labeled polytope of the total space based on the labeled polytopes of the fiber and the base similar to that found in \cite{DmD3,DmD}.
\newline\indent The paper is organized as follows. In the subsequent section we recall some basic facts about smooth and symplectic orbifolds leading to the foundational results of \cite{LT} concerning symplectic toric orbifolds (such as the convexity Theorem and an orbifold version of the Delzant correspondence). The main purpose of this section is to establish notation and make the article more self-contained. With similar purpose in mind we devote section $3$ to present labeled projective spaces which constitute an interesting subclass of symplectic toric orbifolds containing complex projective spaces and weighted projective spaces as special cases (cf. \cite{AbreuToricOrbifolds}). Section $4$ contains our main general results (with proves) concerning toric orbifold coverings and toric orbifold bundles. The remainder of the paper is devoted to applications. We warm up in section $5$ by proving that there is a unique toric structure on each labeled projective space. In section $6$ we extend this result to products of labeled projective spaces of dimension at least $4$. The argument here uses our results on toric orbifold bundles and toric orbifold coverings as well as a repurposing of an argument from \cite{DmD}. The final section of this article treats the remaining cases (i.e. when at least one of the labeled projective spaces is a $(p,q)$-football). Aside from an orbifold version of the corresponding results from \cite{DmD} (i.e. the proof of uniqueness under some restrictions on the symplectic form) we also show that stronger results can be obtained in the presence of singularities with locally maximal rank of the isotropy group.

\section{Symplectic toric orbifolds}
The purpose of this section is to recall elementary information concerning symplectic toric orbifolds and to establish notation in the process.
\newline\indent Let $X$ be a topological space, and $U\subset X$ be its open subset. An {\it orbifold chart} $(\widetilde{U}, H,\phi)$ consists of a connected open set $\widetilde{U}\subset \mathbb{R}^n$, a finite group $H$ acting smoothly and effectively on $\widetilde{U}$ and a continuous $H$-invariant map $\phi \colon \widetilde{U} \to U$, which induces the homeomorphism between $\widetilde{U}/H$ and $U$. An {\it embedding} $\lambda \colon (\widetilde{U}_1,H_1,\phi_1) \hookrightarrow (\widetilde{U}_2,H_2,\phi_2)$ of orbifold charts is a smooth embedding $\lambda \colon \widetilde{U}_1 \hookrightarrow \widetilde{U}_2$, which satisfies $\phi_2 \circ \lambda = \phi_1$. An {\it orbifold atlas} on $X$ is a collection $\mathcal{A}=\{(\widetilde{U}_i,H_i,\phi_i)\}_{i\in I}$ of orbifold charts which satisfies the following condition. For any $i,j\in I$, and any $x \in U_i\cap U_j$, there exists a neighborhood $U_k$ of $x$, such that there is an orbifold chart $(\widetilde{U}_k,H_k,\phi_k)$ of $U_k$, together with embeddings $(\widetilde{U}_k,H_k,\phi_k) \hookrightarrow (\widetilde{U}_i,H_i,\phi_i)$ and $(\widetilde{U}_k,H_k,\phi_k) \hookrightarrow (\widetilde{U}_j,H_j,\phi_j)$. We say that an atlas $\mathcal{A}$ {\it refines} the atlas $\mathcal{B}$, if each chart from $\mathcal{A}$ admits an embedding to some chart of $\mathcal{B}$. Two atlases are called equivalent, if they admit a common refinement. Similar to the manifold case, each orbifold atlas is always contained in a unique maximal one. 

\begin{defi}
An $n$-dimensional smooth orbifold $X$ consists of a Hausdorff paracompact topological space $|X|$ and an {\it orbifold structure}, which is the equivalence class of orbifold atlases $[\mathcal{A}]$. We say that an orbifold chart is a {\it chart} on $X$ if it is contained in some atlas from $[\mathcal{A}]$. Note that if in an atlas $\mathcal{A}$ all groups $H_i$ are trivial, then $\mathcal{A}$ is an atlas of a manifold structure. To simplify notation we shall refer to the underlying topological space by $X$ as well.
\end{defi}
\begin{defi}
Take an orbifold $X$ and some point $x\in X$ with a chart $(\widetilde{U},H,\phi)$ and $x=\phi(\widetilde{x})\in U$. A {\it local group} $\Gamma_x$ is the isomorphism class of the isotropy subgroup $H_{\widetilde{x}}<H$ and it doesn't depend on the choice of a chart. We say that a point is singular if it has non-trivial isotropy group and it is regular otherwise.
\end{defi}
Next we are going to need a notion of morphism of orbifolds. While different definitions of morphisms are considerate (depending on the context) we use the following definition from \cite{IO}.
\begin{defi}
Let $X$ and $Y$ be two orbifolds. Then a smooth map $f \colon X\to Y$ consists of a continuous map $f$ and smooth local lifts at each $x\in X$. More precisely we have:

\begin{itemize}
    \item charts $(\widetilde{U},\Gamma_x,\phi)$ and $(\widetilde{V}, \Gamma_{f(x)}, \psi)$ around $x$ and $f(x)$ respectively, such that $f(U)\subset V$,
    
    \item a homomorphism $\bar{f}_x\colon \Gamma_x \to \Gamma_{f(x)}$,
    
    \item a smooth $\bar{f}_x$-equivariant map $\widetilde{f}_x\colon \widetilde{U}\to \widetilde{V}$ such that the following diagram commutes:
    \begin{center}
    \begin{tikzcd}
    \widetilde{U} \arrow[r,"\widetilde{f}_x"] \arrow[d,"\phi"] & \widetilde{V} \arrow[d, "\psi"] \\
    U \arrow[r,"f"] & V
    \end{tikzcd}
    \end{center}
\end{itemize}
A smooth orbifold map $f\colon X \to Y$ is called a {\it diffeomorphism} if it admits a smooth inverse.
\end{defi}
A special case of a smooth map is given by the following definition which will be of use later on:
\begin{defi} Let $X$ be a smooth orbifold. An orbifold covering of $X$ is a pair $(\hat{X},\rho)$ where $\hat{X}$ is a smooth orbifold and $\rho:\hat{X}\to X$ is a smooth map satisfying:
\begin{enumerate}
    \item For each $x\in X$ there is a chart $(\tilde{U},G,\phi)$ over $x$ such that $\rho^{-1}(U)$ is a disjoint union of open subsets $V_i$
    \item Each $V_i$ admits a chart of the form $(\tilde{U},G_i,\phi_i)$ with $G_i\subset G$ such that $\rho$ lifts to the identity on $\tilde{U}$ with inclusion as the corresponding group homomorphism.
\end{enumerate}
\end{defi}
Most of the well known constructions from differential geometry readily generalize to orbifolds. Starting from the tangent bundle on the orbifold chart $\tilde{U}$ the group action extends via derivative to its total space and then by taking the quotient with respect to this action one gets the tangent "bundle" $TU$ over $U$.
\begin{rem} It is important to note that this object is not a bundle in the classical sense since the fibers over singular points are only quotients of vector spaces. The general framework for dealing with this sort of objects is that of a orbifold vector bundle (see \cite{BG,IO} for details). 
\end{rem}
One can then use the embeddings of charts defined earlier to glue $TU_i$ together getting the orbifold tangent bundle $TX$ as a result. It is easy to generalize this construction to get various tensor bundles as well as their sections. In particular the smooth sections of the $k$-th exterior power of $T^*X$ are called differential $k$-forms. Moreover, since the operator $d$ commutes with the group action in all orbifold charts it is well defined on differential $k$-forms on the orbifold. This allows us to define the orbifold de Rham cohomology in the standard fashion. Finally, the following version of the de Rham Isomorphism Theorem holds: 
\begin{tw}
Let $X$ be an orbifold. Then orbifold de Rham cohomology is isomorphic to singular cohomology of the underlying topological space: $$H^i_{\text{dR}}(X)\simeq H^i(X,\mathbb{R})$$
\end{tw}
\begin{rem} For details of the above construction as well as further results concerning the differential topology and differential geometry of orbifolds see \cite{BG,IO}.
\end{rem}
With the above understood we can now specialize to symplectic and toric orbifolds:
\begin{defi} A symplectic orbifold is a pair $(X,\omega)$, where $\omega$ is a closed non-degenerate $2$-form. A smooth action of a Lie group $G$ on $X$ is called Hamiltonian if there is a $G$-equivariant smooth map $\mu \colon X\to \mathfrak{g}^*$, such that for each $\xi\in \mathfrak{g}$ we have that $d\langle \mu, \xi \rangle = \iota_{\xi^{\#}} \omega$ (where $\xi^{\#}$ denotes the corresponding fundamental vector field). A compact symplectic orbifold $(X,\omega)$ of dimension $2n$ is called toric, if it is equipped with an effective Hamiltonian action of a torus $\mathbb{T}^n$ which admits a moment map $\mu$. We say that two toric orbifolds are isomorphic if there exists a $\lambda$-equivariant symplectomorphism between them (with respect to some automorphism $\lambda$ of the torus). 
\end{defi}
\begin{rem} In the literature there is some ambiguity concerning the previous definition. It is sometimes required that the moment map be fixed which also gives rise to a stronger notion of isomorphism by demanding that the symplectomorphism preserves the chosen moment map. However, for our purposes this definition is more suitable as we will consider in later section the number of symplectic structures on labeled projective spaces up to isomorphism in the above sense. Consequently whenever we refer to the number of toric structures (or the uniqueness of toric structure) on a symplectic manifold we will mean the number of isomorphism classes (similarly to the results in \cite{DmD}).
\end{rem}
In the work of Lerman and Tolman \cite{LT} it is shown that an analogue of Atiyah convexity theorem and Delzant theorem holds for toric orbifolds. 
\begin{defi}
A rational labeled polytope is a simple rational polytope $\Delta\subset\mathbb{R}^n$ together with positive integers $m_F$ (called a label) assigned to each facet $F$. Two rational labeled polytopes $M_1$ and $M_2$ are called equivalent if they differ by a translation preserving the labels.
\end{defi}

\begin{tw}[Convexity for toric orbifolds]
Let a torus act on a compact connected symplectic orbifold $(X,\omega)$ with a moment map $\mu\colon X\to \mathfrak{t}^*$. The image of the moment map $\mu(X)$ is a rational convex polytope. Moreover, the image of $\mu$ is a convex hull of images of points fixed by the torus.
\end{tw}
\begin{tw}[Delzant correspondence for orbifolds]
\label{DelzantOrbifolds}
Given a simple rational polytope $\Delta\subset \mathbb{R}^n$, with a label $m_F$ attached to each facet $F$ of $\Delta$, there exists a toric orbifold $(X,\omega)$ together with a chosen moment map $\mu$, such that $\mu(X)=\Delta$ and the orbifold structure groups of points which map to the interior of facet $F$ are $\mathbb{Z}/m_F\mathbb{Z}$. Moreover, if two labeled rational polytopes are equivalent, their associated toric orbifolds are isomorphic. 
\end{tw}
We will now present a sketch of Delzant construction for orbifolds for the sake of establishing notation for later. Every labeled polytope $\Delta$ can be written as a system of inequalities $\langle x, m_r \eta_r\rangle \leq \kappa_r$, where $\eta_r$ is a primitive vector normal to the $r$-th facet, and $m_r$ is a label associated to the $r$-th facet. Suppose that $\Delta$ has a total of $d$ facets. Consider a linear map $\beta \colon \mathbb{R}^d \to \mathbb{R}^n$, given by $\beta(e_i)=m_i\eta_i$. Denote by $\mathfrak{k}$ the kernel of $\beta$, and by $K$ the kernel of $\beta$ seen as a map from $\mathbb{T}^d$ to $\mathbb{T}^n$. Consider $\mathbb{C}^d$ with a standard action of $\mathbb{T}^d$ and a moment map $$\phi_0(z_1,\ldots, z_d)=\frac{1}{2}(|z_1|^2,\ldots, |z_d|^2)-
(\kappa_1,\ldots,\kappa_d).$$ Then $K$ acts on $\mathbb{C}^d$ with a moment map $\phi=\iota^* \circ \phi_0$, where $\iota$ is an inclusion $\mathfrak{k}\to \mathbb{R}^d$. The symplectic toric orbifold $(X_{\Delta},\omega_{\Delta})$ associated to the labeled rational polytope $\Delta$ is the symplectic reduction $X_{\Delta}=Z/K$, where $Z=\phi^{-1}(0)$.

\begin{rem}[cf. \cite{LT}]
The isotropy and orbifold structure groups of points on a toric orbifold $X$ can be calculated from the labeled polytope $\Delta$. Namely, for each point $p\in X$, denote by $\mathcal{F}(p)$ the set of facets of $\Delta$ containing $\mu(p)$. Then the isotropy group of $p$ is a subtorus $H_p\subset \mathbb{T}^n$, whose Lie algebra is generated by normals $\eta_r$ of facets in $\mathcal{F}(p)$. The orbifold structure group $\Gamma_p$ is isomorphic to $\Lambda_p/\widetilde{\Lambda_p}$, where $\Lambda_p$ is a lattice of circle subgroups in $H_p$ and $\widetilde{\Lambda_p}$ is a sublattice generated by $\{m_r\eta_r\}_{r\in \mathcal{F}(p)}$.
\end{rem}
We finish this section by recalling some similarities between symplectic toric orbifolds and complex toric orbifolds (treated as varieties). Firstly, let us point out that in the final section of \cite{LT} two fundamental results were proved:
\begin{tw} Every compact symplectic toric orbifold carries a complex structure which makes it into a K\"{a}hler orbifold. Moreover, it carries a structure of a complex toric orbifold with fan equal to the fan of its labeled polytope. 
\end{tw}
\begin{tw} Two compact symplectic toric orbifolds are biholomorphic if and only if the fans of their corresponding labeled polytopes coincide.
\end{tw}
It is important to note that the notions of analytic/algebraic complex orbifolds and smooth complex orbifolds do not coincide. This subtlety comes from the fact that $\mathbb{D}^2\slash\mathbb{Z}_p$ does not poses a singularity in the analytic sense. This means that (perhaps somewhat counter intuitively) an analytic complex orbifold admits more then one smooth complex orbifold structure (differing by the codimension $1$ singular sets). See Chapter $4$ of \cite{BG} for a thorough exposition comparing the complex analytic and smooth cases.  In the toric setting this is plainly visible since the data of codimension $1$ singularities corresponds to the labels which do not pass to the fan of the polytope. Motivated by these considerations we shall refer to the orbifold corresponding to the initial labeled polytope with all the labels replaced by $1$ as the {\it underlying complex orbifold} of the given symplectic toric orbifold. The reason for this is that this is the smooth complex orbifold structure such that its singular set coincides with the underlying analytic orbifold (and hence in a sense it is the smooth complex orbifold structure "closest" to the underlying complex analytic orbifold).
\newline\indent Finally, we note that by passing through the construction of the fan associated to the polytope and the corresponding toric variety we can recreate the underlying complex orbifold (albeit without a chosen symplectic structure) as $\mathcal{U}\slash K_{\mathbb{C}}$, where $K_{\mathbb{C}}$ is the complexification of $K$ (from the Delzant construction) and:
$$\mathcal{U}:=\{z\in\mathbb{C}^N\text{ }|\text{ } \bigcap\limits_{i\in I_z} F_i\neq\emptyset\},$$
for $I_Z$ equal to the set of indices $i\in\{1,...,N\}$ such that the $i$-th coordinate $z_i$ of $z$ vanishes.

\section{Labeled projective spaces}
In this section we recall the definition of labeled projective spaces which will become our main object of study in later sections. We follow \cite{AbreuToricOrbifolds} in our exposition with the slight exception that our definition of labeled projective space is more general (allowing different last conormal vector $\eta_{n+1}=\mathbf{v}$).

\begin{defi}
A labeled projective space $(SP^n_{\mathbf{v},\mathbf{w}},\omega_{\mathbf{v},\mathbf{w}})$ for a primitive vector $\mathbf{v}\in \mathbb{N}^n$ and $\mathbf{w}\in \mathbb{N}^{n+1}$ is a toric orbifold associated to a labeled simplex with normal vectors $\eta_i=-e_i$ for $1\leq i \leq n$, and $\eta_{n+1}=\mathbf{v}$ given by the inequalities:
$$\langle x,\eta_i \rangle\leq 0, \quad \text{for } i\in\{1,...,n\},$$
$$\langle x,\eta_{n+1}\rangle \leq 1,$$
with labels given by the set $\mathbf{w}$.
\end{defi}

In order, to better understand which orbifolds appear in the family of labeled projective spaces, we recall some subfamilies.
\begin{egg}[Weighted Projective Spaces]
Fix some $(n+1)$-tuple of positive integers $\lambda=(\lambda_0,\ldots , \lambda_n)$, which satisfy $\gcd(\lambda_0,\ldots,\lambda_n)=1$. Consider a weighted action of $\mathbb{C}^*$ on $\mathbb{C}^{n+1}\setminus \{0\}$ given by $t\cdot (z_0,\ldots, z_n)=(t^{\lambda_0}z_0,\ldots, t^{\lambda_n}z_n)$. Note that this is an almost free action, thus the quotient $(\mathbb{C}^{n+1}\setminus\{0\})/\mathbb{C}^*$ has an orbifold structure. We denote it by $W\mathbb{C}P^n_{\lambda}$ and call it a {\it weighted projective space}. Note that in the special case $\lambda=(1,\ldots,1)$ we get a regular complex projective space, and if $\lambda=(p,q)$ we get a $(p,q)$-football. The symplectic form on $W\mathbb{C}P^n_{\lambda}$ is defined in the same way as for the complex projective space. The action of a torus $\mathbb{T}^n$ on $W\mathbb{C}P^n_{\lambda}$ is given by the usual formula $$(t_1,\ldots,t_n)\cdot [z_0:\ldots:z_n]=[z_0:t_1^{-1}z_1:\ldots :t_n^{-1}z_n].$$ with the moment map 

$$\mu([z_0:\ldots:z_n])=\frac{1}{2}\left(\lambda_1\frac{|z_1|^{\frac{2}{\lambda_1}}}{\Vert z\Vert_{\lambda}^2},\ldots, \lambda_n\frac{|z_n|^{\frac{2}{\lambda_n}}}{\Vert z\Vert_{\lambda}^2}\right)$$

where $\Vert z \Vert_{\lambda}^2=|z_0|^{\frac{2}{\lambda_0}}+\ldots + |z_n|^{\frac{2}{\lambda_n}}$. The image of this moment map in $(\mathbb{R}^n)^*$ can be described by inequalities $\langle x, e_i \rangle \geq 0$ and $\langle x, (\lambda_1^{-1},\ldots,\lambda_n^{-1}) \rangle \leq \frac{1}{2}$. Thus the primitive normal vector to the slanted facet is given by $\eta_{n+1}=\text{lcm}(\lambda_1,\ldots,\lambda_n)\cdot(\lambda_1^{-1},\ldots,\lambda_n^{-1})$. Now the preimage of the interior of each facet $F_i$ of $\text{im }\mu$ is given by $$\mu^{-1}(\text{int } F_i)=\{[z_0:\ldots:z_n] \, | \, z_i=0 \text{ and } z_k\neq 0 \text{ for } k\neq i\}$$

the local orbifold group at each such point is isomorphic to $\mathbb{Z}_{l_i}$ (this is the stabilizer of the corresponding orbit in $\mathbb{C}^{n+1}$), where $l_i=\gcd(\lambda_0, \ldots, \Hat{\lambda_i}, \ldots,\lambda_n)$. Thus the set of labels is given by $\mathbf{w}=(l_0,\ldots,l_n)$. Hence, $W\mathbb{C}P^n_{\lambda}$ is indeed a labeled projective space. Notice that if, for example $\lambda_0=1$, all labels become $1$, but we still can get a non-trivial orbifold structure, since the smoothness condition on the polytope might be violated. 
\end{egg}

\begin{egg}[Orbifold Projective Spaces]
Let's first note that for each weighted projective space there is a natural map $\pi_{\lambda}\colon W\mathbb{C}P^n_{\lambda}\to \mathbb{C}P^n$ given by:

$$\pi_{\lambda}([z_0:\cdots:z_n])=[z_0^{\widehat{\lambda}_0}:\cdots:z_n^{\widehat{\lambda}_n}],$$
where $\widehat{\lambda}_i=\prod_{j\neq i} \lambda_j$. Let $\widehat{\lambda}=\prod_{j} \lambda_j$ and consider a group:

$$\Gamma_{\lambda}=(\mathbb{Z}_{\widehat{\lambda}_0}\times \cdots \times \mathbb{Z}_{\widehat{\lambda}_n})\slash \mathbb{Z}_{\widehat{\lambda}},$$
where $\mathbb{Z}_{\widehat{\lambda}}\hookrightarrow \mathbb{Z}_{\widehat{\lambda}_0}\times \cdots \times \mathbb{Z}_{\widehat{\lambda}_n}$ is given by:
$$\zeta \mapsto (\zeta^{\lambda_0},\cdots,\zeta^{\lambda_n}),$$

using the representation of cyclic group as roots of unity. The group $\Gamma_{\lambda}$ acts by multiplication on $W\mathbb{C}P^n_{\lambda}$. A simple calculation shows that $\pi_{\lambda}([z])=\pi_{\lambda}([z'])$ if and only if $[z]=[\eta]\cdot [z']$ for some $\eta\in \Gamma_{\lambda}$. Thus $\pi_{\lambda}$ factors as:
$$\begin{tikzcd}
W\mathbb{C}P^n_{\lambda} \arrow[rr, "{\pi_{\lambda}}"] \arrow[dr] & & \mathbb{C}P^n \\
& W\mathbb{C}P^n_{\lambda}\slash\Gamma_{\lambda} \arrow[ur,"{[\pi_{\lambda}]}"]
\end{tikzcd}$$
Since $\Gamma_{\lambda}$ is finite, the quotient $O\mathbb{C}P^n_{\lambda}:=W\mathbb{C}P^n_{\lambda}\slash\Gamma_{\lambda}$ has an orbifold structure. We call it an {\it orbifold projective space}. The map $[\pi_{\lambda}]$ is then a homeomorphism, but it is not an orbifold diffeomorphism, since $O\mathbb{C}P^n_{\lambda}$ has a non-trivial orbifold structure. In order to describe its orbifold groups we go back to the weighted action of $\mathbb{C}^*$ on $\mathbb{C}^{n+1}$ and consider its finite extension. 
\newline\indent Let $K^{\mathbb{C}}_{\lambda}$ be a complex Lie group defined by:
$$K^{\mathbb{C}}_{\lambda}=(\mathbb{Z}_{\widehat{\lambda}_0}\times \cdots \times \mathbb{Z}_{\widehat{\lambda}_n}\times \mathbb{C}^*) \slash\mathbb{Z}_{\widehat{\lambda}},$$
where:
$$\mathbb{Z}_{\widehat{\lambda}}\hookrightarrow \mathbb{Z}_{\widehat{\lambda}_0}\times \cdots \times \mathbb{Z}_{\widehat{\lambda}_n}\times \mathbb{C}^*,$$
is given by $\zeta \mapsto (\zeta^{\lambda_0},\cdots, \zeta^{\lambda_n},\zeta^{-1})$.
Then $K^{\mathbb{C}}_\lambda$ acts effectively on $\mathbb{C}^{n+1}$ by:
$$[\eta,t]\cdot (z_0,\cdots,z_n)=(\eta_0 t^{\lambda_0}z_0,\cdots, \eta_n t^{\lambda_n} z_n) \quad \text{for } \eta \in \mathbb{Z}_{\widehat{\lambda}_0}\times \cdots \times \mathbb{Z}_{\widehat{\lambda}_n}, \, t\in \mathbb{C}^*.$$

Because of the exact sequence: $$1\to \mathbb{C}^* \hookrightarrow K^{\mathbb{C}}_{\lambda} \to \Gamma_{\lambda} \to 1$$ we have that: $$(\mathbb{C}^{n+1}\setminus \{0\})/K^{\mathbb{C}}_\lambda \simeq W\mathbb{C}P^n_{\lambda}/\Gamma_{\lambda}=O\mathbb{C}P^n_{\lambda}.$$

Now the local orbifold groups on $O\mathbb{C}P^n_{\lambda}$ can be found from the isotropy groups of the $K^{\mathbb{C}}_{\lambda}$ action. Consider open projective hyperplanes in $O\mathbb{C}P^n_{\lambda}$:
$$H_i^{\circ}=\{[z_0:\cdots:z_n]\in O\mathbb{C}P^n_{\lambda} \, | \, z_i=0 \text{ and } z_k\neq 0 \text{ if } k\neq i\}.$$
\begin{lem}
The local orbifold group $\Gamma_z$ of $z\in H_i^{\circ}$ is isomorphic to $\mathbb{Z}_{\widehat{\lambda_i}}$.
\end{lem}
\begin{proof}
The local orbifold group of $z=[z_0:\cdots:z_n]$ is the same as the isotropy group of $(z_0,\ldots,z_n)\in\mathbb{C}^{n+1}$ under the $K^{\mathbb{C}}_\lambda$. We have that $[\eta,t]\cdot z=z$ if and only if $\eta_k=t^{-\lambda_k}$ for all $k\neq i$. Since $\eta_k$ is $\widehat{\lambda}_k$-th root of unity, $t$ is $\widehat{\lambda}$-th root of unity. Thus we have that $$\Gamma_z\simeq (\mathbb{Z}_{\widehat{\lambda}_i}\times \mathbb{Z}_{\widehat{\lambda}})\slash ((\zeta^{\lambda_i},\zeta^{-1}), \, \zeta\in \mathbb{Z}_{\widehat{\lambda}})\simeq \mathbb{Z}_{\widehat{\lambda}_i}$$ where the last isomorphism is given by $[\eta_i,\zeta]\mapsto \eta_i\cdot \zeta^{\lambda_i}$.
\end{proof}
\end{egg}

\section{Toric orbifold coverings and bundles}
In this section we study Toric orbifold coverings. We start by noting that if $p:\tilde{X}\to X$ is an orbifold covering of a toric orbifold $X$ then $p^*\omega$ and the lifts of fundamental vector fields make $\tilde{X}$ into a symplectic orbifold with a Hamiltonian action of $\mathbb{R}^n$. Moreover, this lifts the torus action on $X$ to an action of $\mathbb{R}^n$ on $X$ and the two actions are connected via a morphism $\lambda:\mathbb{R}^n\to\mathbb{T}^n$ with respect to which $p$ is equivariant. In order to remain contained in the toric category we are interested in the particular case when the action of $\mathbb{R}^n$ on $\tilde{X}$ descends to an effective torus action. This justifies the following definition of toric orbifold coverings:
\begin{defi} A toric orbifold covering $(p,\lambda)$ consists of an orbifold covering $p:\tilde{X}\to X$ together with an epimorphism of Lie groups $\lambda:\mathbb{T}^n\to\mathbb{T}^n$ such that $(\tilde{X},\tilde{\omega})$ is a toric orbifold and $p$ is a $\lambda$-equivariant map (with respect to the torus actions) satisfying $p^*\omega=\tilde{\omega}$.
\end{defi}

It is easy to see that if $p: M\to N$ is a toric covering between compact connected toric manifolds then it is an isomorphism (since such manifolds are well known to be simply connected). This is not the case for orbifolds and orbifold coverings as illustrated by the following simple example.
\begin{egg} Consider $\mathbb{S}^2$ with its standard toric structure and take $\mathbb{Z}_2\subset\mathbb{S}^1$. Then $p:\mathbb{S}^2\to\mathbb{S}^2\slash\mathbb{Z}_2$ is a toric $2$-fold orbifold covering of $\mathbb{S}^2\slash\mathbb{Z}_2$ with the action of $\mathbb{S}^1\slash\mathbb{Z}_2\cong\mathbb{S}^1$.
\end{egg}
More generally, given a toric orbifold $X$ one can always construct a toric orbifold covering by choosing a discreet subgroup $G\subset\mathbb{T}^n$ and considering the quotient map $p: X\to X\slash G$ (equipped with the natural action of $\mathbb{T}^n\slash G\cong\mathbb{T}^n$). We shall now prove the main result of this section which asserts that all toric orbifold coverings are of the above form.
\begin{tw}\label{Covering} Let $p:\tilde{X}\to X$ be a toric orbifold covering of $X$ with $\tilde{X}$ connected. Let $G\subset\mathbb{T}^n$ denote the kernel of the corresponding endomorphism of $\mathbb{T}^n$. Then $X$ is isomorphic to $\tilde{X}\slash G$ via an isomorphism $\phi:X\to\tilde{X}\slash G$, such that $\phi\circ p=\pi$ (where $\pi$ denotes the quotient map).
\begin{proof} It is apparent that $p$ descents to a toric orbifold covering $\overline{p}:\tilde{X}\slash G\to X$ such that each orbit covers its image orbit exactly once. Hence, it suffices to prove that the preimage of an orbit $\mathbb{T}^nx$ through $\overline{p}$ consists of a single orbit (if so then $\overline{p}$ is an isomorphism and $\phi=\overline{p}^{-1}$).
\newline\indent Let $\tilde{\Delta}$ and $\Delta$ be the corresponding polytopes of $\tilde{X}\slash G$ and $X$ respectively. If $\mathbb{T}^nx$ corresponds to a boundary point of $\Delta$ then arbitrary close there is an orbit corresponding to an interior point of $\Delta$  such that its preimage via $p$ has the same number of connected components and hence without loss of generality we can assume that $\mathbb{T}^nx$ corresponds to an interior point of $\Delta$. Moreover, we note that for dimensional reasons the orbits in the interior $\tilde{\Delta}^{\circ}$ of $\tilde{\Delta}$ map precisely to orbits in the interior $\Delta^{\circ}$ of $\Delta$. Hence, since $\overline{p}$ is a cover such that each orbit covers its image orbit exactly once we conclude that the correspondence between orbits and points in polygons induce a covering $p_{\Delta}:\tilde{\Delta}^{\circ}\to \Delta^{\circ}$. However, due to convexity of these polytopes this covering has to be $1$-fold and so no two orbits corresponding to points in $\tilde{\Delta}^{\circ}$ map to the same orbit in $X$, which finishes the proof.
\end{proof}
\end{tw}
This phenomenon is useful in the context of the so called orbifold bundles (or $V$-bundles, see \cite{BG} for a more detailed exposition). A simple yet convincing argument for studying such objects is the fact that in general the tangent spaces of an orbifold do not form a vector bundle but an orbifold vector bundle since the fibers over singularities are in general quotients of vector spaces. For a general orbifold bundle we allow the isotropy group in the uniformization charts to act on the fibers which potentially makes the fibers over singularities, spaces finitely covered by the general fiber.
\begin{defi}\label{Uniformizing system} Let $\pi:E\to U$ be a continuous map between topological spaces. A symplectic uniformizing system for $E$ over $U$ is the following collection of data:
\begin{enumerate}
\item A symplectic orbifold chart $(V,G,\phi)$ for $U$.
\item A triple $(V\times F,G,\tilde{\phi})$, where $F$ is a symplectic orbifold, such that $G$ acts on $V\times F$ by $g(x,v)=(gx,\rho(x,g)v)$ for some smooth function $\rho:V\times G\to Symp(F)$ satisfying $\rho(gx,h)\rho(x,g)=\rho(x,gh)$ and $\tilde{\phi}:V\times F\to E$ is a $G$-invariant map inducing a diffeomorphism between $(V\times F)\slash G$ and $E$ (endowing the latter space with a symplectic structure).
\item If $\tilde{\pi}$ is the projection $V\times F\to V$ then $\phi\circ\tilde{\pi}=\pi\circ\tilde{\phi}$.
\end{enumerate}
\end{defi}
\begin{rem} We note that for later application we have slightly generalized the standard definition in order to allow regular fibers to be orbifolds (compare with \cite{BG}).
\end{rem}
We say that two symplectic uniformizing systems $(V,G,\phi)$ and $(V',G',\phi')$ for $E$ over $U$ are isomorphic if there is a group isomorphism $\lambda:G\to G'$ and a $\lambda$-equivariant diffeomorphism $\psi:V\times F\to V'\times F$ which is fiberwise a symplectomorphism and satisfies $\phi=\phi'\circ\psi$.
\begin{defi}\label{Orbifold bundle} Let $\pi: E\to X$ be a smooth map between symplectic orbifolds $E$ and $X$. A structure of a symplectic orbifold bundle with general fiber $F$ on $\pi: E\to X$ is given by assigning to each point $x\in X$ a uniformizing system on $\pi|_{\pi^{-1}(U_x)}:\pi^{-1}(U_x)\to U_x$, such that for each $y\in U_x$ there is a neighbourhood $U_{xy}\subset U_x\cap U_y$ such that the uniformization on $U_{xy}$ induced by the ones on $U_x$ and $U_y$ are isomorphic. We say that a symplectic orbifold bundle is a symplectic bundle of orbifolds if it admits an orbifold bundle structure such that the action of the isotropy groups on the fibers in the uniformizing charts is trivial.
\end{defi}
\begin{defi} We say that a symplectic orbifold bundle $\pi: E\to X$ with general fiber $F$ is toric if:
\begin{enumerate}
\item $F$, $E$ and $X$ are toric orbifolds of dimensions $2k$, $2n+2k$ and $2n$ respectively.
\item There is a short exact sequence:
$$\begin{tikzcd}
0\arrow[r]&{T}^{k}\arrow[r,"i"]&\mathbb{T}^{n+k}\arrow[r,"p"]&\mathbb{T}^{n}\arrow[r]&0,
\end{tikzcd}$$
such that the orbifold covering of fibers is equivariant with respect to $i$ and $\pi$ is equivariant with respect to $p$. We say that a toric orbifold bundle is a toric bundle of orbifolds if it is a symplectic bundle of orbifolds.
\end{enumerate}
\end{defi}
\begin{rem} In the above definition we mention the orbifold covering of fibers to underline the fact that the general fiber $F$ is not diffeomorphic to every fiber since it only covers the singular fibers.
In what follows we will adopt the standard abuse of language and refer to $F$ as the fiber of the orbifold bundle, while when treating the preimages of points we will speak of regular fibers or singular fibers (depending on whether the points considered are regular or singular).
\end{rem}
We are now ready to provide a description of a toric bundle of orbifolds.
\begin{defi}
Let $\widetilde{\Delta}=\bigcap_{i=1}^{\widetilde{N}} \{x\in \widetilde{\mathfrak{t}}^* \, | \, \langle \widetilde{\eta_i} , x \rangle \leq \widetilde{\kappa_i}\}$ and $\widehat{\Delta}=\bigcap_{i=1}^{\widehat{N}} \{x\in \widehat{\mathfrak{t}}^* \, | \, \langle \widehat{\eta_i} , x \rangle \leq \widehat{\kappa_i}\}$ be two standard representations of labeled polytopes. The labeled polytope $\Delta\subset \mathfrak{t}^*$ is called a labeled bundle polytope with fiber $\widetilde{\Delta}$ and base $\widehat{\Delta}$ if there is a short exact sequence $$0\to \widetilde{\mathfrak{t}} \xrightarrow{\iota} \mathfrak{t} \xrightarrow{\pi} \widehat{\mathfrak{t}} \to 0$$

so that 

\begin{itemize}
    \item $\Delta$ is combinatorialy equivalent to $\widetilde{\Delta}\times \widehat{\Delta}$.
    \item If $\widetilde{\eta_i}'$ is an outward conormal vector to the facet $\widetilde{F_i}'$ of $P$ corresponding to $\widetilde{F_i} \times \widehat{\Delta}\subset \widetilde{\Delta}\times \widehat{\Delta}$, then $\widetilde{\eta_i}'=\iota(\widetilde{\eta_i})$ for all $1\leq i \leq \widetilde{N}$.
       \item If $\widehat{\eta_i}'$ is an outward conormal vector to the facet $\widehat{F_i}'$ of $P$ corresponding to $\widetilde{\Delta} \times \widehat{F_i}\subset \widetilde{\Delta}\times \widehat{\Delta}$, then $\pi(\widehat{\eta_i}')=\widehat{\eta_i}$ for all $1\leq i \leq \widehat{N}$.
       \item The labels of the facets corresponding to $\widetilde{\Delta} \times \widehat{F_i}$ and $\widetilde{F_i} \times \widehat{\Delta}$ are equal to those of $\widehat{F_i}$ and $\widetilde{F_i}$ respectively.
\end{itemize}
\end{defi}
  By the same arguments as used in \cite{DmD3} and an elementary computation of the corresponding labels (or alternatively as a consequence of Theorem \ref{boclass} again together with a computation of labels) the following propositions is true:
  \begin{prop}\label{bclass} The labeled polytope of the total space $(E,\omega_E)$ of a toric bundle with fiber $(F,\omega_F)$ over $(X,\omega_X)$ is a toric bundle polytope over the labeled polytope of $X$ with the labeled polytope of $F$ as its fiber. Moreover, if the labeled polytope of a toric orbifold is a toric bundle polytope over the labeled polytope of $(X,\omega_X)$ with the labeled polytope of $(F,\omega_F)$ as its fiber then the orbifold is a total space of a toric bundle over $(X,\omega'_X)$ with fiber $(F,\omega'_F)$ for some choices of symplectic structure on $X$ and $F$ making them into symplectic toric manifolds.
  \end{prop}
  Similarly as in the manifold case the situation greatly simplifies when both the fiber and the base are rational simplices.
  So suppose we have that a Delzant polytope $\Delta$ is a bundle with base $\Delta_{k_1}$ and fiber $\Delta_{k_2}$ which are rational simplices. Now, since $\Delta$ is combinatorialy equivalent to $\Delta_{k_1}\times \Delta_{k_2}$, it has exactly $k_1+k_2+2$ facets, from which $k_1+1$ are of the form $F_i'=\Delta_{k_2} \times F_i$ and $k_2+1$ are of the form $G_i'=G_i\times \Delta_{k_1}$, where $F_i$ are the facets of $\Delta_{k_1}$ and $G_i$ are the facets of $\Delta_{k_2}$. Identify $\mathbb{R}^{k_1+k_2}=\mathbb{R}^{k_1} \times \mathbb{R}^{k_2}$. Then from the definition of a polytope bundle we can easily obtain normal vectors to faces $F_i'$, which are just $\eta_i=-e_i$ for $1\leq i\leq k_1$, and $\eta_{k_1+1}=\sum_{i=1}^{k_1} b_ie_i$ where $b_i$ are the linear coordinates of the conormal to the slanted face of the rational simplex $\Delta_{k_1}$. Now let's find the next $k_2$ vectors $\eta_{k_1+j+1}$ for $1\leq j\leq k_2$, which project onto the basis $\{-e_1,...,-e_{k_2}\}$ of $\mathbb{R}^{k_2}$. All facets corresponding to these normal vectors have a common adjacent vertex, which is the intersection $\bigcap_{i=1}^{k_1} F_i' \cap \bigcap_{j=1}^{k_2} G_j'$. The smoothness at that point now ensures that the first $k_1$ coordinates of $\eta_{k_1+j+1}$ are all $0$. Thus we get that $\eta_{k_1+j+1}=-e_{k_1+j+1}$ for $1\leq j \leq k_2$. Now the last normal vector $\eta_{k_1+k_2+2}$ must be of the form $(a_1,\ldots , a_{k_1}, b'_1, \ldots, b'_{k_2})$ for some integer $k_1$-tuple $(a_1,\ldots, a_{k_1})$, where $b'_j$ are the linear coordinates of the conormal to the slanted face of the rational simplex $\Delta_{k_2}$. Thus the bundle involving simplices is determined by one slanted face, which is given by an integer $k_1$-tuple $(a_1,\ldots,a_{k_1})$.
  \newline\indent We shall now study general toric orbifold bundles in order to give a similar description. Let us start by giving the following simple example which will guide our efforts. 
  \begin{egg} Let $X$ be $(\mathbb{S}^2\times\mathbb{S}^2)\slash\mathbb{Z}_2$ where $\mathbb{Z}_2$ acts by rotating both spheres by $\pi$ around an axis. By considering $\mathbb{S}^2$ with its standard toric structure we endow $X$ with the structure of a toric orbifold. Moreover, by splitting the second copy of $\mathbb{S}^2$ into hemispheres we can make $X$ into a toric orbifold $\mathbb{S}^2$-bundle over the $(2,2)$-football. We now consider how this affects the polytope. We start with the unit square $\tilde{\Delta}$ which is the polytope $\mathbb{S}^2\times\mathbb{S}^2$. Hence, the corresponding homomorphism of tori induces a linear map on the Lie algebra given in the basis $\{(1,1),(0,1)\}$ by doubling the first element. Writing this in the standard basis we get the matrices:
  $$
  A=\begin{bmatrix} 2&0\\1&1
  \end{bmatrix}\quad and \quad A^*=\begin{bmatrix} 2&1\\0&1
  \end{bmatrix}.
  $$
  By applying $(A^*)^{-1}$ to the vertices of the square we get the vertices of the polytope $\Delta$ of $X$. Namely, $(0,0)$, $(\frac12,0)$,$(-\frac12,1)$ and $(0,1)$.
  \end{egg}
 Comparing this to the trivial $\mathbb{S}^2$-bundle over a $(2,2)$-football we see that it exchanged the two facets with labels $2$ for two slanted facets without labels. Hence, it seems that the change needed in Definition \ref{bclass} is to change the labels to divisors $b_i$ of the original label $a_i$ of the basis and change condition $3)$ to $\pi(\widehat{\eta_i}')=\frac{a_i}{b_i}\widehat{\eta_i}$. Hence, we arrive at the following definition justified by the subsequent Theorem:
 \begin{defi}
Let $\widetilde{\Delta}=\bigcap_{i=1}^{\widetilde{N}} \{x\in \widetilde{\mathfrak{t}}^* \, | \, \langle \widetilde{\eta_i} , x \rangle \leq \widetilde{\kappa_i}\}$ and $\widehat{\Delta}=\bigcap_{i=1}^{\widehat{N}} \{x\in \widehat{\mathfrak{t}}^* \, | \, \langle \widehat{\eta_i} , x \rangle \leq \widehat{\kappa_i}\}$ be two standard representations of labeled polytopes. The labeled polytope $\Delta\subset \mathfrak{t}^*$ is called a labeled orbifold bundle polytope with fiber $\widetilde{\Delta}$ and base $\widehat{\Delta}$ if there is a short exact sequence $$0\to \widetilde{\mathfrak{t}} \xrightarrow{\iota} \mathfrak{t} \xrightarrow{\pi} \widehat{\mathfrak{t}} \to 0$$

so that 

\begin{itemize}
    \item $\Delta$ is combinatorialy equivalent to $\widetilde{\Delta}\times \widehat{\Delta}$.
    \item If $\widetilde{\eta_i}'$ is an outward conormal vector to the facet $\widetilde{F_i}'$ of $P$ corresponding to $\widetilde{F_i} \times \widehat{\Delta}\subset \widetilde{\Delta}\times \widehat{\Delta}$, then $\widetilde{\eta_i}'=\iota(\widetilde{\eta_i})$ for all $1\leq i \leq \widetilde{N}$.
       \item The labels of the facets corresponding to $\widetilde{F_i} \times \widehat{\Delta}$ are equal to those of $\widetilde{F_i}$.
       \item The labels $b_i$ of the facets corresponding to $\widetilde{\Delta} \times \widehat{F_i}$ are divisors of the labels $a_i$ of $\widehat{F_i}$.
              \item If $\widehat{\eta_i}'$ is an outward conormal vector to the facet $\widehat{F_i}'$ of $P$ corresponding to $\widetilde{\Delta} \times \widehat{F_i}\subset \widetilde{\Delta}\times \widehat{\Delta}$, then $\pi(\widehat{\eta_i}')=\frac{a_i}{b_i}\widehat{\eta_i}$ for all $1\leq i \leq \widehat{N}$.
\end{itemize}
\end{defi}
 \begin{tw}\label{boclass} The labeled polytope of the total space $(E,\omega_E)$ of a toric orbifold bundle with fiber $(F,\omega_F)$ over $(X,\omega_X)$ is a toric orbifold bundle polytope over the labeled polytope of $X$ with the labeled polytope of $F$ as its fiber. Moreover, if the labeled polytope of a toric orbifold is a toric orbifold bundle polytope over the labeled polytope of $(X,\omega_X)$ with the labeled polytope of $(F,\omega_F)$ as its fiber then the orbifold is a total space of a toric orbifold bundle over $(X,\omega'_X)$ with fiber $(F,\omega'_F)$ for some choices of symplectic structure on $X$ and $F$ making them into symplectic toric orbifolds.
 \begin{proof} Throughout the proof let us denote by $\Delta$, $\widetilde{\Delta}$ and $\widehat{\Delta}$ the labeled polytopes of $E$, $F$ and $X$ respectively.
 \newline\indent We start by proving that the labeled polytope of a toric orbifold bundle is a toric orbifold bundle polytope. Firstly, let us note that the exact sequence:
 $$0\to \widetilde{\mathfrak{t}} \xrightarrow{\iota} \mathfrak{t} \xrightarrow{\pi} \widehat{\mathfrak{t}} \to 0,$$
 arises naturally from the short exact sequence of torii from the definition of toric orbifold bundles.
 \newline\indent Secondly, we prove combinatorial equivalence. Note that on the open set $U$ of $X$ consisting of all the points with trivial isotropy group (of the action) the bundle $E|_{U}$ is trivial. This implies that the set of points in $E|_{U}$ fixed under the action $\mathbb{T}^k\subset\mathbb{T}^k\times\mathbb{T}^n$ (corresponding to $\iota(\tilde{\mathfrak{t}})\subset \mathfrak{t}$) has precisely $m$ connected components $\widetilde{U_i}$ (where $m$ is the number of fixed points on a generic fiber). Due to our description of toric coverings each singular fiber has the same number of fixed points as $F$. Hence, by studying the uniformizations around fixed points of $X$ we can infer that the closures $\overline{U}_i$ of $\widetilde{U_i}$ in $E$ are all pairwise disjoint. The combinatorial equivalence now follows by assigning to a vertex in $\widetilde{\Delta}\times\widehat{\Delta}$ corresponding to a point $(y,x)\in F\times X$ the vertex corresponding to the fixed point in the fiber of $E$ over $x$ which lies in the appropriate $\overline{U}_i$ (corresponding to $y$ via the chosen trivialization of $E|_{U}$).
 \newline\indent The second and third condition in the definition of toric orbifold bundle polytopes is easy to see by studying an orbifold chart around a point corresponding to a generic point in $\widetilde{F_i} \times \widehat{\Delta}$ such that it is a product of charts on $F$ and $X$. From this one can see that the corresponding orbifold isotropy group and action isotropy group has to be the same as for the fiber (since the projection of such a point is regular in $X$).
 \newline\indent For the final two conditions, let us consider a local uniformizing system for $E$ around a point $x\in X$ corresponding to a generic point in $\widehat{F_i}$:
\begin{center}
    \begin{tikzcd}
    V\times F \arrow[r,"\widetilde{\pi}"] \arrow[d,"\tilde{\phi}"] & V \arrow[d, "\phi"] \\
    V\times F\slash G \arrow[r,"\pi"] & V\slash G
    \end{tikzcd}
    \end{center}
 Firstly, let us note that since the conormal vector $\widehat{\eta_i}'$ corresponding to $\widetilde{\Delta} \times \widehat{F_i}$ is tangent to the kernel of the action restricted to the manifold correspodning to it then its image has to fix $\widehat{F_i}$. Consequently, $\pi(\widehat{\eta_i}')$ must be some multiplicity of $\widehat{\eta_i}$. Also from this uniformizing system we can see that $\mathbb{Z}_{b_i}$ which is the orbifold isotropy group at $x$ acts on the fiber over $x$. Then, the isotropy group of a generic point in the fiber over $x$ is precisely the kernel of this action on the fiber which is isomorphic to $\mathbb{Z}_{a_i}$ for some $a_i$ dividing $b_i$, which in turn implies that $a_i$ is the label of the corresponding facet of $\Delta$. Finally, by taking $\overline{\eta}_i$ to be the conormal vector corresponding to the facet in the uniformizing system $V\times F$ and using the fact that the projection of the uniformizing system onto $V$ gives a chart on $X$ we conclude that $\overline{\eta}_i$ corresponds via $\tilde\pi\circ\phi$ to $a_i\widehat{\eta}_i'$ in $\widehat{\Delta}$  while its projection to $V\times F\slash G$ corresponds to $b_i\widehat{\eta}_i$, which proves that $\pi(\widehat{\eta}_i')=\frac{b_i}{a_i}\widehat{\eta}_i$ since the projection in the uniformizing system and in the bundle commute with the uniformizing map and the chart.
 \newline\indent To prove the inverse implication we proceed similarly to Remark $5.2$ of \cite{DmD3}. For $\Delta$ (resp. $\widetilde{\Delta}$, $\widehat{\Delta}$) choose $\mathcal{U}$ (resp. $\widetilde{\mathcal{U}}$, $\widehat{\mathcal{U}}$) and $K$ (resp. $\widetilde{K}$, $\widehat{K}$) as in the orbifold Delzant construction so that $E\cong\mathcal{U}\slash K_{\mathbb{C}}$ (resp. $F\cong\widetilde{\mathcal{U}}\slash \widetilde{K_{\mathbb{C}}}$, $X\cong\widehat{\mathcal{U}}\slash \widehat{K_{\mathbb{C}}}$). By Identyfying $\mathbb{C}^{\widetilde{N}}\times\mathbb{C}^{\widehat{N}}$ with $\mathbb{C}^{\widetilde{N}+\widehat{N}}$ and using the combinatorial equivalence of $\widetilde{\Delta}\times\widehat{\Delta}$ with $\Delta$ we see that $\mathcal{U}=\widetilde{\mathcal{U}}\times\widehat{\mathcal{U}}$.
 \newline\indent We shall now prove that $\widehat{K}\cong K\slash\widetilde{K}$. Firstly, let us notice that since $\widetilde{\eta}_j'=\iota(\widetilde{\eta}_j)$ we have $\widetilde{K}=\mathbb{T}^{\widetilde{N}}\cap K\subset K\subset\mathbb{T}^{\widetilde{N}}\times\mathbb{T}^{\widehat{N}}$. Moreover, due to the commutativity of the bottom right square in the diagram below $K\subset \mathbb{T}^{\widetilde{N}}\times\mathbb{T}^{\widehat{N}}$ maps to $\widehat{K}\subset \mathbb{T}^{\widehat{N}}$ through the projection $\Pi_2$ onto the second factor. This allows us to consider the commutative diagram:
$$\begin{tikzcd}
&0\arrow[d]&0\arrow[d]&0\arrow[d]&\\
0\arrow[r]& \widetilde{K}\arrow[r]\arrow[d]&\mathbb{T}^{\widetilde{N}}\arrow[r,"\widetilde{p}"]\arrow[d,"I_1"]&\mathbb{T}^{\widetilde{n}}\arrow[r]\arrow[d,"\iota_1"]&0\\
0\arrow[r]& K\arrow[r]\arrow[d]&\mathbb{T}^{\widetilde{N}}\times\mathbb{T}^{\widehat{N}}\arrow[r,"p"]\arrow[d,"\Pi_2"]&\mathbb{T}^{\widetilde{n}}\times\mathbb{T}^{\widehat{n}}\arrow[r]\arrow[d,"\pi_2"]&0
\\0\arrow[r]&\widehat{K}\arrow[d]\arrow[r]&\mathbb{T}^{\widehat{N}}\arrow[d]\arrow[r,"\widehat{p}"]&\mathbb{T}^{\widehat{n}}\arrow[d]\arrow[r]&0\\
&0&0&0&,
\end{tikzcd}$$
where $\pi_2$ denotes the projection onto the second factor and $I_1$ and $\iota_1$ are respective inclusions of the first factor. The final non-zero arrow in each row is the map from the Delzant construction. Since we know that all the rows and all but the first column are exact we can use the nine lemma for abelian groups to prove the exactness of the first column.
\newline\indent From this we get:
$$E\cong\mathcal{U}\slash K_{\mathbb{C}}\cong (\widetilde{\mathcal{U}}\times\widehat{\mathcal{U}})\slash K_{\mathbb{C}}\cong ((\widetilde{\mathcal{U}}\slash \widetilde{K_{\mathbb{C}}})\times\widehat{\mathcal{U}})\slash (K_{\mathbb{C}}\slash\widetilde{K_{\mathbb{C}}})\cong (F\times\widehat{\mathcal{U}})\slash\widehat{K_{\mathbb{C}}}.$$

Using the Slice Theorem to the $\widehat{K}_{\mathbb{C}}$-action on $\widehat{\mathcal{U}}$ we get a slice $S$ through any point $\widehat{x}\in\widehat{\mathcal{U}}$. Due to the fact that $X\cong\widehat{\mathcal{U}}\slash\widehat{K_{\mathbb{C}}}$ is an orbifold of the same dimension as $S$, we get that the stabilizer group of $S$ is finite and consequently $F\times S$ provides the desired uniformization around $x\in X\cong\widehat{\mathcal{U}}\slash\widehat{K_{\mathbb{C}}}$ corresponding to $\widehat{x}$.
 \end{proof}
 \end{tw}
\section{Uniqueness of toric structures on labeled projective spaces}
We firstly show that the number of faces of a labeled polytope depends only on the topology of the corresponding orbifold. To this end let us first recall the following well known definition:
\begin{defi}
\label{fhvec}
The $f$-vector of a rational labeled polytope $\Delta$ is an integer vector composed of numbers $(f_0,\ldots,f_n)$, where $f_i$ is a number of $i$-dimensional faces of $\Delta$. We set $f_n=1$. The $f$-polynomial of $\Delta$ is given by $F(t)=\sum_i f_i t^i$.

Now in many cases it is more convenient to use another vector, called the $h$-vector of $\Delta$. Its polynomial is given by $H(t)=F(t-1)=\sum_i h_i t^i$, and we can calculate that:

$$h_k=\sum_{i\geq k} f_i (-1)^{i-k} {i\choose k}$$

On the other hand, since $F(t)=H(t+1)$, we can derive the following inverse formula: $$f_i=\sum_{k\geq i} h_k{k\choose i}$$
\end{defi}
\begin{prop}
\label{bettiorbifold}
Define the Betti numbers of an orbifold $X$ as $b_i:=\dim H_i(X,\mathbb{Q})$. Then if $X$ is a toric orbifold with labeled rational polytope $\Delta$, then $b_{2i}=h_i(\Delta)$ and $b_{2i+1}=0$ for all $i\in\mathbb{Z}$.
\begin{proof} The proof essentially follows its counterpart in the manifold case (see \cite{T}) using a generalization of known results from Morse theory to orbifolds (see \cite{LT}). First by \cite{LT}, Corollary 6.4 we obtain that around a fixed point $x$ there is a chart $(\widetilde{U},\Gamma_x,\phi)$, such that the smooth local lift of $\mu$ is given by $\widetilde{\mu}(x_1,y_1,\ldots,x_n,y_n)=\sum_i \lambda^{(i)}(x_i^2+y_i^2)$, where $\lambda^{(i)}$ are the primitive integer vectors generating the edges. Then by taking $X\in t$ which generates a dense subgroup of $\mathbb{T}^n$ we get a Morse function $\mu^X=\langle X,\mu \rangle$ (see also \cite{LT} Lemma 5.3) such that its critical points are precisely the fixed point of the action. Then using the above local description of $\mu$ as well as orbifold versions of Morse inequalities (\cite{LT} Theorem 4.7) we get that odd Betti numbers are zero, and $dim (C_{2k}(\mu^X))=b_{2k}(M)$ (where $C_{2k}(\mu^X)$ denotes the corresponding Morse complex). Now for every vertex $p$ of $\Delta$ define $\text{ind}_X(p)$ as a number of edges going down from that vertex (meaning that the linear functional $l_X(Y)=\langle Y,X \rangle$ is decreasing on those edges). Using that $l_X\circ\mu=\mu^X$ we get that the critical points of $\mu^X$ with index $2i$ correspond to vertices with index $i$ and hence:
$$b_{2i}(M)=B(X,i):=\# \{ p \, | \, \text{ind}_X(p)=i\}.$$
 
 Now note that each vertex $p$ of index $m$ is a maximum of $l_X$ restricted to some face of dimension $\leq m$. Now let's calculate the number $f_i$ the following way. For every $i$-dimensional face of $\Delta$, assign its vertex, on which $l_X$ is maximal. Then we will get a vertex of index $m\geq i$. But each such vertex has exactly $m \choose i$ $i$-faces going down from it. Hence in order to calculate $f_i$, we need to calculate the amount of vertices of index $m\geq i$ and each such vertex must be considered $m \choose i$ times. Thus we get the formula: $$f_i=\sum_{m\geq i} B(X,m) {m \choose i}.$$
 Comparing this formula to Definition \ref{fhvec}, we see that $B(X,m)=h_m$ (since the numbers satisfying the above property can be uniquely determined from the $f$-vector).
\end{proof}
\end{prop}

\begin{cor}
\label{fvectors}
If $\Delta_0$ and $\Delta_1$ are two labeled polytopes of two toric structures on the same orbifold $(X,\omega)$, then the $f$-vector of $\Delta_0$ is the same as the $f$-vector of $\Delta_1$.
\begin{proof}
Note that the $f$-vectors of two polytopes are equal if and only if their $h$-vectors are equal. But by the above proposition the $h$-vector of the polytope depends only on topology of $M$, and not on the toric structure.
\end{proof}
\end{cor}
Next we establish the uniqueness of toric structures on labeled projective spaces. To this end we will need the following well known fact:
\begin{lem}
\label{isotropygroups}
Suppose that $A\in M_n(\mathbb{Z})$ is a non-singular matrix with integer entries. Then $|\mathbb{Z}^n\slash A\mathbb{Z}^n|=|\det(A)|$.
\begin{proof}
Consider the Smith normal form $D=SAT^{-1}$, where $S,T\in SL_n(\mathbb{Z})$ and $D$ is diagonal. Then $\text{Im}(A)=\text{Im}(SA)=\text{Im}(TD)$. The last image is generated by the vectors $d_1v_1,\ldots,d_nv_n$, where $v_1,\ldots v_n$ are the vector entries of $T$, and thus form a basis, and $d_1,\ldots, d_n$ are the diagonal entries of $D$. So we have that $$|\mathbb{Z}^n/A\mathbb{Z}^n|=|\mathbb{Z}^n/\langle d_1v_1,\ldots,d_nv_n \rangle|=|\mathbb{Z}_{d_1}\oplus \cdots \oplus \mathbb{Z}_{d_n}|=|d_1\ldots d_n|=|\det(A)|.$$
\end{proof}
\end{lem}
With this we are ready to present the main result of this section.
\begin{tw}
\label{UniqueSP}
There is a unique toric structure on each labeled projective space $(\mathbb{S}P^n_{\mathbf{v},\mathbf{w}},\lambda\omega_{\mathbf{v},\mathbf{w}})$.
\begin{proof}
Suppose that $\Delta$ is a rational labeled polytope coming from some toric structure $(\mathbb{S}P^n_{\mathbf{v},\mathbf{w}},\lambda\omega_{\mathbf{v},\mathbf{w}})$. Then by Corollary \ref{fvectors} we have that $f_0(\Delta)=f_0(\Delta_n)=n+1$. This means that $\Delta$ is also a simplex. Since the isotropy groups of generic points corresponding to facets doesn't depend on the toric structure, we get that the labels of $\Delta$ have to coincide with $\mathbf{w}$. Now take the labels $(w_1,\ldots,w_{n})$ which corresponded to facets adjacent to $0$ in the original simplex. The vertex $p$ in $\Delta$ adjacent to $(F_1,\ldots,F_{n})$ must have the same isotropy group as the $0$ vertex for geometric reasons. By the remark after Theorem \ref{DelzantOrbifolds} this isotropy group is $\mathbb{Z}_{w_1}\times \ldots \times \mathbb{Z}_{w_{n}}$. Then by Lemma \ref{isotropygroups} we obtain that normal vectors $\eta_1,\ldots,\eta_{n}$ to $F_1,\ldots,F_{n}$ must form a $\mathbb{Z}$-basis of $\mathbb{Z}^n$, otherwise the orders of isotropy groups will not match. Then we can apply a $SL_n(\mathbb{Z})$ transformation to $\Delta$, to translate those normal vectors to the basis $-e_1,\ldots, -e_n$ of $\mathbb{Z}^n$. 
\newline\indent Now the last normal vector $\eta_{n+1}$ can also be computed by calculating isotropy groups. Take a vertex $p_i$ of $\Delta$ adjacent to all facets except for $F_i$. From the original simplex, the order of isotropy group at $p_i$ must be $$|\mathbb{Z}^n/\langle w_1e_1, \ldots , \widehat{e}_i,\ldots, w_ne_n, w_{n+1}\mathbf{v}\rangle|=\det( w_1e_1, \ldots , \widehat{e}_i,\ldots, w_ne_n, w_{n+1}\mathbf{v})=\prod_{j\neq i}w_j \cdot v_i$$ where $v_i$ is the $i$-th coordinate of $\mathbf{v}$. On the other hand if we calculate the isotropy group from $\Delta$ directly, we will get $\prod_{j\neq i}w_j \cdot (\eta_{n+1})_i$. This means that $i$-th coordinate of $\eta_{n+1}$ is equal to $v_i$. By taking $i$ from $1$ to $n$ we get that $\eta_{n+1}=\mathbf{v}$.
\end{proof}
\end{tw}
\section{Uniqueness of toric orbifold structures on products of labeled projective spaces}
We start by proving the following proposition which can be viewed as a generalization of Lemma 4.10 from \cite{DmD3}.
\begin{prop}\label{Sbundle} Let $X$ be a toric $2n$-orbifold such that the underlying symplectic orbifold is $(\mathbb{S}P^{k_1}_{\mathbf{v}_1,\mathbf{w}_1}\times\mathbb{S}P^{k_2}_{\mathbf{v}_2,\mathbf{w}_2},\lambda_1\omega_{\mathbf{v}_1,\mathbf{w}_1}+\lambda_2\omega_{\mathbf{v}_2,\mathbf{w}_2})$, where $\lambda_1,\lambda_2\in\mathbb{R}\backslash\{0\}$. Then $X$ is isomorphic (as a toric orbifold) to a toric bundle with fiber $\mathbb{S}P^{k_1}_{\mathbf{v}_1,\mathbf{w}_1}$ over $\mathbb{S}P^{k_2}_{\mathbf{v}_2,\mathbf{w}_2}$ or viceversa.
\begin{proof} Let $\Delta$ be the corresponding labeled polytope of the toric structure (more precisely of some moment map for the toric structure). The number of facets of $\Delta$ is given by $f_{n-1}(\Delta)=n+2$ and hence by Proposition 1.1.1 from \cite{T} $\Delta$ is combinatorially equivalent to a product of simplices. Moreover, the rational cohomology of the resulting space is $H^{\bullet}(\mathbb{S}P^{k_1}_{\mathbf{v}_1,\mathbf{w}_1},\mathbb{Q})\otimes H^{\bullet}(\mathbb{S}P^{k_2}_{\mathbf{v}_2,\mathbf{w}_2},\mathbb{Q})$ which is isomorphic (as vector spaces) to $H^{\bullet}(\mathbb{C}P^{k_1},\mathbb{Q})\otimes H^{\bullet}(\mathbb{C}P^{k_2},\mathbb{Q})$. Using the correspondence between the $h$-vector of $\Delta$ and the Betti numbers of $X$ we can conclude that it is in fact combinatorially equivalent to a product of a $k_1$-simplex and a $k_2$-simplex. More precisely, if it is equivalent to a product of a $l_1$-simplex and a $l_2$-simplex (with $l_1+l_2=k_1+k_2$ and $l_1\notin\{k_1,k_2\}$) then its cohomology is isomorphic (as graded vector spaces) to $H^{\bullet}(\mathbb{C}P^{l_1},\mathbb{Q})\otimes H^{\bullet}(\mathbb{C}P^{l_2},\mathbb{Q})$ which is not isomorphic to $H^{\bullet}(\mathbb{C}P^{k_1},\mathbb{Q})\otimes H^{\bullet}(\mathbb{C}P^{k_2},\mathbb{Q})$ which gives a contradiction.
\newline\indent The setup in the previous paragraph allows us to denote the facets of $\Delta$ by $F_1 , . . . , F_{k_1+1}$ and $F'_1, . . . , F'_{k_2+1}$, in such a way that $F_I \cap F'_J\neq \emptyset$ for any proper subsets $I \subset \{1, . . . , k_1 + 1\}$ and $J \subset \{1, . . . , k_2 + 1\}$. Denote the corresponding conormal vectors by $\eta_1 , . . . ,\eta_{k_1+1}$ and $\eta'_1 , . . . ,\eta'_{k_2+1}$. Given $i\in\{1, . . . , k_1 + 1\}$ and $j\in\{1, . . . , k_2 + 1\}$, let $I_i =\{1, . . . , k_1 + 1\}\backslash\{i\}$ and $J_j =\{1, . . . , k_2 + 1\}\backslash\{j\}$. Since $F_{I_{k_1+1}} \cap F_{J_{k_2+1}}\neq \emptyset$, the vectors $\eta_1 , . . . ,\eta_{k_1}, \eta'_1 , . . . ,\eta'_{k_2}$ form a basis for $\mathfrak{t}^{k_1+k_2}$. We can use this basis to write the remaining two conormal vectors:
$$\eta_{k_1+1}=\sum\limits_{i=1}^{k_1} a_i\eta_i+\sum\limits_{j=1}^{k_2} a'_j\eta'_j$$
$$\eta'_{k_2+1}=\sum\limits_{i=1}^{k_1} b_i\eta_i+\sum\limits_{j=1}^{k_2} b'_j\eta'_j$$
Given $i\in\{1, . . . , k_1 + 1\}$ and $j\in\{1, . . . , k_2 + 1\}$, let $A_{i,j}$ denote the matrix with columns $\eta_1,...,\eta_{i-1},\eta_{i+1},...,\eta_{k_1+1},\eta'_1,...,\eta'_{j-1},\eta'_{j+1},...,\eta'_{k_1+1}$.
\newline\indent Since $F_{I_i}\cap F'_{J_j}$ is not empty it defines a vertix $p_{i,j}$ with $|det(A_{i,j})|$ equal to the rank of the orbifold isotropy group at the corresponding fixed point $x_{i,j}$ in the underlying complex orbifold of $\mathbb{S}P^{k_1}_{\mathbf{v}_1,\mathbf{w}_1}\times\mathbb{S}P^{k_2}_{\mathbf{v}_2,\mathbf{w}_2}$. In particular, all $|det(A_{i,j})|$ depend only on the orbifold structure and not the toric structure. In particular, we can assume that $det(A_{k_1+1,k_2+1})=1$ (by changing the order of $\eta_i$ and $\eta'_j$ if necessary) and hence $\eta_1 , . . . ,\eta_{k_1}, \eta'_1 , . . . ,\eta'_{k_2}$ form a $\mathbb{Z}$-basis. Moreover, this assumption fixes the signs of all $det(A_{i,j})$ since moving along an edge from one vertex to the other (this changes one conormal vector) changes the sign of the determinant (assuming the new vector is put on the same position in the basis). Consequently, after such normalization each $det(A_{i,j})$ is an invariant of the orbifold itself.
\newline\indent On the other hand: $$det(A_{i,k_2+1})=(-1)^{k_1+1-i}a_idet(A_{k_1+1,k_2+1})=(-1)^{k_1+1-i}a_i$$ (via elementary operations). This implies that the coefficients $a_i$ again depend only on the orbifold structure (since all the other constants in the equation depend only on the orbifold structure). Similarily, by studying $det(A_{k_1+1,j})$ we conclude the same for $b'_j$.
\newline\indent Furthermore, by considering: $$det(A_{i,j})=(-1)^{k_1+k_2+2-i-j}(a_ib'_j-b_ia'_j)det(A_{k_1+1,k_2+1})=(-1)^{k_1+k_2+2-i-j}(a_ib'_j-b_ia'_j)$$ we again arrive at the conclusion that $b_ia'_j$ is dependent only on the orbifold structure. Hence, by comparison with the standard structure on $\mathbb{S}P^{k_1}_{\mathbf{v}_1,\mathbf{w}_1}\times\mathbb{S}P^{k_2}_{\mathbf{v}_2,\mathbf{w}_2}$ we get $b_ia'_j=0$ and consequently either all $b_i$ or all $a'_j$ are zero. This proves that in fact the structure given is isomorphic to a bundle of some $k_1$-simplex over a $k_2$-simplex (or viceversa).
\newline\indent We finish the proof by establishing these simplices via the following two observations:
\begin{enumerate}
    \item Since $\eta_1,...\eta_{k_1}$ form a $\mathbb{Z}$-basis for the space $V$ generated by them and by the fact that $a_i$ are invariants of the orbifold we get that $V\cap\Delta$ is in fact the simplex corresponding to the underlying complex orbifold of $\mathbb{S}P^{k_1}_{\mathbf{v}_1,\mathbf{w}_1}$. Same applies for $\eta'_1,...\eta'_{k_2}$ and $\mathbb{S}P^{k_2}_{\mathbf{v}_2,\mathbf{w}_2}$.
    \item The labels on the sides also have to stay the same since they can be computed from the orders of isotropy groups at the singular subspaces of complex codimension $1$.
\end{enumerate}
\end{proof}
\end{prop}
Let us denote by $G_{\mathbf{w}}$ the subgroup of $\mathbb{T}^n$ such that $W\mathbb{C}P^n_{\mathbf{w}}\slash G_{\mathbf{w}}\cong O\mathbb{C}P^n_{\mathbf{w}}$. This group has the property that whenever $\mathbb{S}P^{k_1}_{\mathbf{v}_1,\mathbf{w}_1}$ has the same underlying complex orbifold as $\mathbb{C}P^n_{\mathbf{w}}$ then the underlying complex orbifold of $\mathbb{S}P^{n}_{\mathbf{v}_1,\mathbf{w}_1}\slash G_{\mathbf{w}}$ is simply $\mathbb{C}P^n$ since its polytope (disregarding the labels) is the standard simplex. The main idea behind the proof of uniqueness of toric structures on $\mathbb{S}P^{k_1}_{\mathbf{v}_1,\mathbf{w}_1}\times\mathbb{S}P^{k_2}_{\mathbf{v}_2,\mathbf{w}_2}$ is to consider an appropriate group $G_{\mathbf{\tilde{w}}_1,\mathbf{\tilde{w}}_2}:=G_{\mathbf{\tilde{w}}_1}\oplus G_{\mathbf{\tilde{w}}_2}$ so that (in the case of the standard structure) the underlying complex orbifold of $\mathbb{S}P^{k_1}_{\mathbf{v}_1,\mathbf{w}_1}\times\mathbb{S}P^{k_2}_{\mathbf{v}_2,\mathbf{w}_2}\slash G_{\mathbf{\tilde{w}}_1,\mathbf{\tilde{w}}_2}$ is $\mathbb{C}P^{k_1}\times\mathbb{C}P^{k_2}$, and study the quotient of an arbitrary toric orbifold $X$ which is symplectomorphc to $\mathbb{S}P^{k_1}_{\mathbf{v}_1,\mathbf{w}_1}\times\mathbb{S}P^{k_2}_{\mathbf{v}_2,\mathbf{w}_2}$ by this group.
\begin{lem}\label{Cbundle} Under the above notation $X\slash G_{\mathbf{\tilde{w}}_1,\mathbf{\tilde{w}}_2}$ is isomorphic to a toric bundle of $\mathbb{S}P^{k_1}_{\mathbf{v}_1,\mathbf{w}_1}\slash G_{\mathbf{\tilde{w}}_1}$ over $\mathbb{S}P^{k_2}_{\mathbf{v}_2,\mathbf{w}_2}\slash G_{\mathbf{\tilde{w}}_2}$ or vice versa.
\begin{proof} We start by applying Proposition $\ref{Sbundle}$ to $X$ to conclude that (without loss of generality) the resulting orbifold is isomorphic to a bundle of $\mathbb{S}P^{k_1}_{\mathbf{v}_1,\mathbf{w}_1}$ over $\mathbb{S}P^{k_2}_{\mathbf{v}_2,\mathbf{w}_2}$. Due to the form of each $G_{\mathbf{w}}$ the covering $\pi_{\mathbb{T}^{k_1+k_2}}:\mathbb{T}^{k_1+k_2}\to\mathbb{T}^{k_1+k_2}\slash G_{\mathbf{\tilde{w}}_1,\mathbf{\tilde{w}}_2}$ is simply given by a covering on each canonical circle. Hence, the resulting induced map on Lie algebras written in the canonical basis is just a diagonal matrix $diag(c_1,...,c_{k_1+k_2})$. From this we conclude that the polytope $\Delta_{\mathbf{w}_1,\mathbf{w}_2}$ of $X\slash G_{\mathbf{\tilde{w}}_1,\mathbf{\tilde{w}}_2}$ is given by rescaling for each $i\in\{1,...,k_1+k_2\}$ the direction of $e^*_i$ by $\frac{1}{c_i}$. Hence, we see that the result is also a bundle while due to the definition of $G_{\mathbf{\tilde{w}}_1,\mathbf{\tilde{w}}_2}$ the faces contained in $\mathfrak{t}^{k_1}$ and $\mathfrak{t}^{k_2}$ are precisely the standard simplices. Moreover, it is easy to see that the labels of these polytopes have to agree with the labels on the polytopes of $\mathbb{S}P^{k_1}_{\mathbf{v}_1,\mathbf{w}_1}\slash G_{\mathbf{\tilde{w}}_1}$ and $\mathbb{S}P^{k_2}_{\mathbf{v}_2,\mathbf{w}_2}\slash G_{\mathbf{\tilde{w}}_2}$ respectively.
\end{proof}
\end{lem}
The next step in the proof is the following proposition. The proof is basically a slight repurposing of the ideas in cases 1-3 from section 2.4 of \cite{DmD}.
\begin{prop}\label{coho} Let $\mathbb{S}P^{k_1}_{\mathbf{v}_1,\mathbf{w}_1}$ and $\mathbb{S}P^{k_2}_{\mathbf{v}_2,\mathbf{w}_2}$ for $k_1\geq 1$ and $k_2\geq 2$ be such that the underlying complex orbifolds are $\mathbb{C}P^{k_1}$ and $\mathbb{C}P^{k_2}$ respectively. Moreover, let $X$ be a toric bundle with fiber $\mathbb{S}P^{k_1}_{\mathbf{v}_1,\mathbf{w}_1}$ over $\mathbb{S}P^{k_2}_{\mathbf{v}_2,\mathbf{w}_2}$ such that its cohomology ring contains an embedded (as rings) copy of $H^{\bullet}(\mathbb{C}P^{k_1}\times\mathbb{C}P^{k_2},\mathbb{Z})$ in such a way that:
\begin{enumerate}
    \item The embedding restricted to the second degree is an isomorphism of $\mathbb{Z}$-modules.
    \item The cohomology of $X$ is free of the same rank in each degree as that of $H^{\bullet}(\mathbb{C}P^{k_1}\times\mathbb{C}P^{k_2},\mathbb{Z})$.
\end{enumerate}Then $X$ is isomorphic to $\mathbb{S}P^{k_1}_{\mathbf{v}_1,\mathbf{w}_1}\times\mathbb{S}P^{k_2}_{\mathbf{v}_2,\mathbf{w}_2}$.
\begin{proof} Suppose that this bundle is not trivial, that is the vector $(a_1,\ldots, a_{k_1})$ defined in section $3$ is non-zero. Now as in the proof of Proposition 1.13 of \cite{DmD}, we can suppose that all $a_i\leq 0$. Consider a weighted Stanley-Reisner ring of the polytope $\Delta$ of $X$ (cf. \cite{B}). Since $\Delta$ is smooth, this is just the regular Stanley-Reisner ring of $\Delta$. Thus the cohomology of $H^*(X;\mathbb{Z})$ can be presented as $$H^*(X;\mathbb{Z})\simeq \mathbb{Z}[x_1,\ldots,x_{k_1+k_2+2}]\slash (\mathcal{I}+\mathcal{J})$$ where $\mathcal{I}$ gives us multiplicative relations $\prod_{i\in I} x_i=0$ for each set of indices, such that the intersection of facets $\cap_{i\in I}F_i$ is empty, and $\mathcal{J}$ gives us additive relations $\sum_i \langle \eta_i,e_j \rangle x_i=0$. Now denote $\alpha:=x_{k_1+k_2+2}$. Then from the linear relations for $k_1+1\leq j\leq k_1+k_2$ we get that $x_{k_1+2}=\ldots=x_{k_1+k_2+1}=\alpha$, and by the multiplicative relation we get that $\alpha^{k_2+1}=0$. Similarly for $1\leq j \leq k_1$ we get relations $$-x_j+x_{k_1+1}+a_j\alpha=0, \quad \prod_{j=1}^{k_1+1} x_j=0.$$

Define $\beta:=x_{k_1+1}$. Putting $a_{k_1+1}:=0$ we contemplate those relations into \begin{equation}
\label{relations}
    0=\prod_{i=1}^{k_1+1}(\beta+a_i\alpha)=\beta^{k_1+1}+\sigma_1\beta^{k_1}\alpha+\cdots+\sigma_{k_1}\beta\alpha^{k_1}
\end{equation} where $\sigma_i$ is the value of the $i$-th symmetric polynomial on $(a_1,\ldots,a_{k_1+1})$.

From the linear relations above we see that $\alpha$ and $\beta$ generate $H^2(X;\mathbb{Z})\simeq \mathbb{Z}^2$. The standard basis $(\alpha_0,\beta_0)$ of $H^2(X;\mathbb{Z})$ satisfies $\alpha_0^{k_2+1}=0$ and $\beta_0^{k_1+1}=0$ thus there is a matrix with integer entries $A,B,C,D$ with $AD-BC=1$, which translates $\alpha,\beta$ to the standard basis, so that $$(A\alpha+B\beta)^{k_2+1}=0=(C\alpha+D\beta)^{k_1+1}.$$

We split the consideration into the following cases

\begin{enumerate}
    \item[Case 1] $k_1>k_2$.
    
    In this case $H^{\bullet}(\mathbb{C}P^{k_1}\times\mathbb{C}P^{k_2},\mathbb{Z})\subset H^*(X)$ is freely generated by $\alpha$ and $\beta$ in degrees $\leq 2k_2$. In degree $2k_2+2$ there are two relations $$\alpha^{k_2+1}=0, \quad (A\alpha+B\beta)^{k_2+1}=0.$$ These two relations can't be independent, since otherwise the rank of $H^{2k_2+2}(X)$ would be $k_2$ instead of $k_2+1$. Thus $B=0$ and $A=\pm 1$ (we can suppose that $A=1$). Then $D=1$, so $(C\alpha+\beta)^{k_1+1}=0$. By the same argument this relation must be equivalent to (\ref{relations}) after substituting $\alpha^{k_2+1}=0$. By equating the coefficients we get $\sigma_1=C(k_1+1)$ and $\sigma_2=C^2\frac{k_1(k_1+1)}{2}$. Then $\sum_i a_i^2=\sigma_1^2-2\sigma_2=C^2(k_1+1)$. But by Cauchy-Schwartz $$C^2(k_1+1)^2=\sigma_1^2\leq k_1\sum_i a_i^2=k_1(k_1+1)C^2$$ - we get a contradiction.
    
    \item[Case 2] $k_1<k_2$.
    
    Here the equation $(C\alpha+D\beta)^{k_1+1}=0$ can't be independent with (\ref{relations}). But if $C\neq 0$, then this equation contains $\alpha^{k_1+1}$, while (\ref{relations}) doesn't. So $C=0$, and comparing coefficients we get $\sigma_1=\sum_i a_i=0$. But since $a_i\leq 0$ by construction, we get that all $a_i=0$, a contradiction.
    
    \item[Case 3] $k_1=k_2$.
    
    In this case we have all four of previously considered relations, all in the same degree $2k_1+2$. By rank argument, only two of them must be linearly independent. Since $\alpha^{k_1+1}=0$ is independent from (\ref{relations}), those two must imply $(A\alpha+B\beta)^{k_1+1}=0$ and $(C\alpha+D\beta)^{k_1+1}=0$. Without loss of generality suppose that $A,D\neq 0$. Then the equation $(C\alpha+D\beta)^{k_1+1}$ must be $D^{k_1+1}$ times the (\ref{relations}) by comparing coefficients. Thus as in the first case we get $D\sigma_1=(k_1+1)C$ and $D^2\sigma_2=\frac{k_1(k_1+1)}{2}C^2$, therefore we are in the same situation as in case 1 and we get a contradiction.
\end{enumerate}
\end{proof}
\end{prop}
With this we are finally ready to prove the main result of this section.
\begin{tw}\label{NoSoc} Let $X$ be a toric orbifold symplectomorphic to $\mathbb{S}P^{k_1}_{\mathbf{v}_1,\mathbf{w}_1}\times\mathbb{S}P^{k_2}_{\mathbf{v}_2,\mathbf{w}_2}$ endowed with the symplectic form $\lambda_1\omega_{\mathbf{v}_1,\mathbf{w}_1}+\lambda_2\omega_{\mathbf{v}_2,\mathbf{w}_2}$ for $k_1,k_2\geq 2$. Then $X$ is isomorphic to the standard toric structure on $\mathbb{S}P^{k_1}_{\mathbf{v}_1,\mathbf{w}_1}\times\mathbb{S}P^{k_2}_{\mathbf{v}_2,\mathbf{w}_2}$ with the given symplectic form.
\begin{proof} As before let $G_{\mathbf{\tilde{w}}_1,\mathbf{\tilde{w}}_2}:=G_{\mathbf{\tilde{w}}_1}\oplus G_{\mathbf{\tilde{w}}_2}$ be the group such that the underlying complex orbifold of $\mathbb{S}P^{k_i}_{\mathbf{v}_i,\mathbf{w}_i}\slash G_{\mathbf{\tilde{w}}_i}$ is a manifold. By Lemma \ref{Cbundle} $X\slash G_{\mathbf{\tilde{w}}_1,\mathbf{\tilde{w}}_2}$ is isomorphic to a toric bundle of $\mathbb{S}P^{k_1}_{\mathbf{v}_1,\mathbf{w}_1}\slash G_{\mathbf{\tilde{w}}_1}$ over $\mathbb{S}P^{k_2}_{\mathbf{v}_2,\mathbf{w}_2}\slash G_{\mathbf{\tilde{w}}_2}$.
\newline\indent The next step is to show that the cohomology ring of $X\slash G_{\mathbf{\tilde{w}}_1,\mathbf{\tilde{w}}_2}$ satisfies the assumptions of Proposition \ref{coho}. However, since $X\slash G_{\mathbf{\tilde{w}}_1,\mathbf{\tilde{w}}_2}$ is isomorphic to a toric bundle of $\mathbb{S}P^{k_1}_{\mathbf{v}_1,\mathbf{w}_1}\slash G_{\mathbf{\tilde{w}}_1}$ over $\mathbb{S}P^{k_2}_{\mathbf{v}_2,\mathbf{w}_2}\slash G_{\mathbf{\tilde{w}}_2}$ we can conclude that its cohomology as graded modules agrees with that of $\mathbb{C}P^{k_1}\times\mathbb{C}P^{k_2}$. Hence, we only need to show that the ring structure is the same and hence the second assumption is satisfied. Due to the rational Stanley-Reisner presentations of the cohomology rings of $X$ and $X\slash G_{\mathbf{\tilde{w}}_1,\mathbf{\tilde{w}}_2}$ it is easy to see that the mapping induced by the covering becomes an isomorphism in rational cohomology and hence a monomorphism in integer cohomology (since it is free). Hence, to prove the first assumption it suffices to show that we can choose a basis $(\alpha,\beta)$ of $H^{2}(X\slash G_{\mathbf{\tilde{w}}_1,\mathbf{\tilde{w}}_2},\mathbb{Z})$, such that:
$$\alpha^{k_2+1}=0=\beta^{k_1+1}.$$
Then the ring spammed by these elements would necessarily have to map monomorphicaly to a subring of $H^{\bullet}(X,\mathbb{Z})\cong H^{\bullet}(\mathbb{S}P^{k_1}_{\mathbf{v}_1,\mathbf{w}_1}\times\mathbb{S}P^{k_2}_{\mathbf{v}_2,\mathbf{w}_2},\mathbb{Z})$ generated by the multiples of its standard generators $\{A\tilde{\alpha},B\tilde{\beta}\}$ since these are the only elements satisfying $(A\tilde{\alpha})^{k_2+1}=0=(B\tilde{\beta})^{k_1+1}$, which proves that it is isomorphic to the cohomology ring of $\mathbb{C}P^{k_1+1}\times\mathbb{C}P^{k_2+1}$. To this end we proceed as follows:
\begin{enumerate}
    \item We use the Smith Normal Form Theorem to find a basis $\{v_1,v_2\}$ of $H^2(X\slash G_{\mathbf{\tilde{w}}_1,\mathbf{\tilde{w}}_2},\mathbb{Z})$ such that its image in $H^{2}(X,\mathbb{Z})$ consists of muliplicities $\{A_1w_1,A_2w_2\}$ for some non-zero integers $A_1$ and $A_2$ and a basis $\{w_1,w_2\}$.
    \item We find the $Sl(\mathbb{Z},2)$ transformation changing the basis $(w_1,w_2)$ to $\{\tilde{\alpha},\tilde{\beta}\}$ (we might have to change the sign of $\alpha$ in order for the transformation to be orientation preserving, this is of no consequence to the rest of the proof).
    \item We apply the inverse of this transformation to $\{v_1,v_2\}$ to gain the desired basis $\{\alpha,\beta\}$.
\end{enumerate}
Since as mentioned before the map induced by the covering maps $\{\alpha,\beta\}$ to $\{A\tilde{\alpha},B\tilde{\beta}\}$ and is a monomorphism, the elements $\{\alpha,\beta\}$ satisfy the desired relations. This implies due to Proposition \ref{coho} that $X\slash G_{\mathbf{\tilde{w}}_1,\mathbf{\tilde{w}}_2}$ is isomorphic to $(\mathbb{S}P^{k_1}_{\mathbf{v}_1,\mathbf{w}_1}\slash G_{\mathbf{\tilde{w}}_2})\times (\mathbb{S}P^{k_2}_{\mathbf{v}_2,\mathbf{w}_2}\slash G_{\mathbf{\tilde{w}}_2})$.
\newline\indent Finally, we note that since there is an isomorphism between $X\slash G_{\mathbf{\tilde{w}}_1,\mathbf{\tilde{w}}_2}$ and $(\mathbb{S}P^{k_1}_{\mathbf{v}_1,\mathbf{w}_1}\slash G_{\mathbf{\tilde{w}}_2})\times (\mathbb{S}P^{k_2}_{\mathbf{v}_2,\mathbf{w}_2}\slash G_{\mathbf{\tilde{w}}_2})$ due to the fact that $\Delta_{\mathbf{\tilde{w}}_1,\mathbf{\tilde{w}}_2}$ (the polytope of $X\slash  G_{\mathbf{\tilde{w}}_1,\mathbf{\tilde{w}}_2}$) is already a bundle of appropriate simplices the corresponding special linear transformation is the identity. Hence, $\Delta_{\mathbf{\tilde{w}}_1,\mathbf{\tilde{w}}_2}$ is just $\Delta_{k_1}\times\Delta_{k_2}$ and hence $\Delta$ (the polytope of $X$) is also a product of rational simplices of dimensions $k_1$ and $k_2$, since $\Delta_{\mathbf{\tilde{w}}_1,\mathbf{\tilde{w}}_2}$ arises form $\Delta$ by rescaling the canonical directions (as in the proof of Lemma \ref{Cbundle}).
\end{proof}
\end{tw}
\section{Products of labeled projective spaces with $(p,q)$-footbals}
In this section we continue the study of products of labeled projective spaces by studying the cases when at least one of the integers $k_1$ and $k_2$ is equal to $1$. This case is much more subtle as already for manifolds counterexamples to the uniqueness of toric structures are well known (see for example the discussion after Proposition 1.8 in \cite{DmD}). The purpose of this section is two fold:
\begin{enumerate}
    \item Prove that uniqueness holds for products of $(p,q)$-footbals with labeled projective spaces under some restrictions on $\lambda_1$ and $\lambda_2$ (analogously as is done in \cite{DmD})
    \item Prove that uniqueness for toric structures on the product holds under certain restrictions on the singular set regardless of $\lambda_1$ and $\lambda_2$.
\end{enumerate}
In particular, achievement of the second purpose shows that (contrary to the impression one might have by the contents of the previous section) the number of toric structures is in fact influenced by the labels.
For both these goals the following Lemma (adapted to orbifolds from \cite{DmD} section $2.4$ case $4$) is crucial:
\begin{lem}\label{foot} Let $\pi:X\to\mathbb{S}P^{1}_{\mathbf{v}_1,\mathbf{w}_1}$ be a toric $\mathbb{S}P^{k}_{\mathbf{v}_1,\mathbf{w}_1}$-bundle over $\mathbb{S}P^{1}_{\mathbf{v}_1,\mathbf{w}_1}$ symplectomorphic to $\mathbb{S}P^{1}_{\mathbf{v}_1,\mathbf{w}_1}\times\mathbb{S}P^{k}_{\mathbf{v}_2,\mathbf{w}_2}$ considered with the symplectic form $\lambda_1\omega_{\mathbf{v}_1,\mathbf{w}_1}+\lambda_2\omega_{\mathbf{v}_2,\mathbf{w}_2}$. Moreover, let us assume that one of the edges of the corresponding labeled bundle polytope corresponds to a section $S$ of the bundle with trivial normal bundle (more precisely it is a product of $\mathbb{S}P^{1}_{\mathbf{v}_1,\mathbf{w}_1}$ and the tangent space to the corresponding fixed point in the fiber). Then $X$ is isomorphic to $\mathbb{S}P^{1}_{\mathbf{v}_1,\mathbf{w}_1}\times\mathbb{S}P^{k}_{\mathbf{v}_2,\mathbf{w}_2}$.
\begin{proof} Let $x_0$ be a fixed point of the action contained in $S$. Let us consider $X\slash G_{\mathbf{\tilde{w}}_2}$ for an appropriate $G_{\mathbf{\tilde{w}}_2}\subset\mathbb{T}^k$ (such that the underlying complex orbifold of $\mathbb{S}P^{k}_{\mathbf{v}_2,\mathbf{w}_2}\slash G_{\mathbf{\tilde{w}}_2}$ is $\mathbb{C}P^k$). Moreover, without loss of generality we assume that the polytope of $\mathbb{S}P^{k}_{\mathbf{v}_2,\mathbf{w}_2}\slash G_{\mathbf{\tilde{w}}_2}$ is noramlized at $\pi(x_0)$. By our assumptions, the normal bundle to $\pi(S)$ is a trivial bundle with fiber equal to the tangent cone over $x_0$. Moreover, this bundle splits into trivial complex line orbifold bundles $L_1,...,L_k$ corresponding to $e_1,...,e_k$ (i.e. $L_i$ is fixed by the torus generated by $\{e_1,...,e_{i-1},e_{i+1},...,e_k\}$). In other words $L_i$ are the tubular neighbourhoods of $S$ in the suborbifolds corresponding to the $2$-dimensional faces of the labeled polytope. By comparing these neighbourhoods to their counterparts in the underlying complex orbifolds we see that the vector bundles $\tilde{L}_i$ defined analogously to $L_i$ are all trivial (otherwise an arbitrary section of these bundles would have to intersect $S$ which is a contradiction with the triviality of $L_i$). From this we conclude (similarly as in \cite{DmD}) that each $a_i$ in the description of the bundle via a slanted face vanishes since (due to normalization at $x_0$) they are precisely the Chern classes of $\tilde{L}_i$. This implies that the considered toric structure on $(\mathbb{S}P^{1}_{\mathbf{v}_1,\mathbf{w}_1}\times \mathbb{S}P^{k}_{\mathbf{v}_2,\mathbf{w}_2})\slash G_{\mathbf{\tilde{w}}_2}$ is the standard one.
\newline\indent The proof is concluded by passing to the covering space similarly as in the proof of Theorem \ref{Soc1}.
\end{proof}
\end{lem}
For the first goal we present the following Theorems:
\begin{tw}\label{Soc1} Let $X$ be a toric orbifold symplectomorphic to $\mathbb{S}P^{1}_{\mathbf{v}_1,\mathbf{w}_1}\times\mathbb{S}P^{k}_{\mathbf{v}_2,\mathbf{w}_2}$ endowed with the symplectic form $\lambda_1\omega_{\mathbf{v}_1,\mathbf{w}_1}+\lambda_2\omega_{\mathbf{v}_2,\mathbf{w}_2}$ for $k\geq 2$ and $\lambda_2\geq\lambda_1$. Then $X$ is isomorphic to the standard toric structure on $\mathbb{S}P^{1}_{\mathbf{v}_1,\mathbf{w}_1}\times\mathbb{S}P^{k}_{\mathbf{v}_2,\mathbf{w}_2}$ with the given symplectic form.
\begin{proof} Without loss of generality we assume that $\lambda_2$ is positive. It then suffices to prove the theorem for the normalized form $\omega:=\lambda\omega_{\mathbf{v}_1,\mathbf{w}_1}+\omega_{\mathbf{v}_2,\mathbf{w}_2}$ with $\lambda=\frac{\lambda_1}{\lambda_2}$.
\newline\indent As in the previous section we deduce that $X$ is a toric bundle of $\mathbb{S}P^{1}_{\mathbf{v}_1,\mathbf{w}_1}$ over $\mathbb{S}P^{k}_{\mathbf{v}_2,\mathbf{w}_2}$ or vice versa. Moreover, since Proposition \ref{coho} is true for $k_1=1$ we can use the argument in the proof of Theorem \ref{NoSoc} verbatim to treat the case when $\mathbb{S}P^{1}_{\mathbf{v}_1,\mathbf{w}_1}$ is a fiber (regardless of $\lambda_1$ and $\lambda_2$). Hence, it suffices to examine the case when $\mathbb{S}P^{1}_{\mathbf{v}_1,\mathbf{w}_1}$ is the base space.
\newline\indent Firstly, let us do this under the assumption that the underlying complex orbifold of $\mathbb{S}P^{k}_{\mathbf{v}_2,\mathbf{w}_2}$ is $\mathbb{C}P^k$. Let $a_i$ be as in the description of a toric bundle of orbifolds via a slanted face. We note that each $a_i=0$ since otherwise this would give a new toric structure on the underlying complex orbifold (considered with the symplectic form $\lambda\omega_{1}+\omega_{k}$ where $\omega_{k}$ denotes the standard symplectic structure on $\mathbb{C}P^k$) which is a contradiction with \cite{DmD}.
\newline\indent For the general case, we note that for $\mathbb{S}P^k_{\mathbf{v}_2,\mathbf{w}_2}$ the lowest dimensional strata (in the stratification of the orbifold by singularity type) is a manifold diffeomrphic to $\mathbb{C}P^s$ for some $0\leq s\leq k$. The case $s=0$ is in fact treated in greater generality in Theorem \ref{sing} and hence here we assume $1\leq s$. This implies by section $2.4$ of \cite{DmD} that in the underlying complex orbifold of $X$ the toric structure restricted to $\mathbb{C}P^1\times\mathbb{C}P^s$ (treated as a subset of the underlying complex orbifold) is isomorphic to the standard one. Let $S$ be the image of $\mathbb{C}P^1\times\{x\}$ corresponding to an edge in the polytope of $\mathbb{C}P^1\times\mathbb{C}P^s$. Taking now the rational homology class $[S]$ of $S$ in $\mathbb{C}P^1\times\mathbb{C}P^s$ we see that $[S]=[\mathbb{C}P^1\times\{x\}]$ since a diffeomorphism of $\mathbb{C}P^1\times\mathbb{C}P^s$ has to preserve the standard generators in cohomology up to sign (in order to preserve the ring structure) which in turn is fixed due to the symplectic structure. Hence, using the Hurewicz Theorem they are homotopic in $\mathbb{C}P^1\times\mathbb{C}P^s$. The bundle normal to $S$ splits into the bundle normal to $S$ in $\mathbb{C}P^1\times\mathbb{C}P^s$ (which is trivial due to the above isomorphism) and the bundle normal to $\mathbb{C}P^1\times\mathbb{C}P^s$ restricted to $S$ (which is trivial by the above homotopy). Hence, the normal bundle to $S$ is trivial in the underlying complex orbifold of $X$, implying that the normal bundle to the corresponding manifold $S$ in $X$ is also trivial (via the same argument as in the proof of Lemma \ref{foot}). This concludes the proof by Lemma \ref{foot}. 
\end{proof}
\end{tw}
\begin{tw}\label{11} Let $X$ be a toric orbifold symplectomorphic to $\mathbb{S}P^{1}_{\mathbf{v}_1,\mathbf{w}_1}\times\mathbb{S}P^{1}_{\mathbf{v}_2,\mathbf{w}_2}$ endowed with the symplectic form $\lambda_1\omega_{\mathbf{v}_1,\mathbf{w}_1}+\lambda_2\omega_{\mathbf{v}_2,\mathbf{w}_2}$ for $\lambda_2=\lambda_1$. Then $X$ is isomorphic to the standard toric structure on $\mathbb{S}P^{1}_{\mathbf{v}_1,\mathbf{w}_1}\times\mathbb{S}P^{1}_{\mathbf{v}_2,\mathbf{w}_2}$ with the given symplectic form.
\begin{proof} In this case the underlying complex orbifold is simply the product of spheres and hence the theorem follows immediately from its analogue in \cite{DmD}.
\end{proof}
\end{tw}
Next we prove that if $\mathbb{S}P^{k}_{\mathbf{v}_2,\mathbf{w}_2}$ is sufficiently singular then the restrictions on $\lambda_1$ and $\lambda_2$ are superfluous.
\begin{tw}\label{sing} Let $X$ be a toric orbifold symplectomorphic to $\mathbb{S}P^{1}_{\mathbf{v}_1,\mathbf{w}_1}\times\mathbb{S}P^{k}_{\mathbf{v}_2,\mathbf{w}_2}$ with $k\geq 2$, endowed with the symplectic form $\lambda_1\omega_{\mathbf{v}_1,\mathbf{w}_1}+\lambda_2\omega_{\mathbf{v}_2,\mathbf{w}_2}$. Furthermore, assume that $\mathbb{S}P^{k}_{\mathbf{v}_2,\mathbf{w}_2}$ has a singular point $x_0$ such that there exists a neighbourhood $U$ of $x_0$ such that the ranks of isotropy groups of points in $U\backslash\{x_0\}$ are strictly smaller than the rank of the isotropy group of $x_0$. Then $X$ is isomorphic to the standard toric structure on $\mathbb{S}P^{1}_{\mathbf{v}_1,\mathbf{w}_1}\times\mathbb{S}P^{k}_{\mathbf{v}_2,\mathbf{w}_2}$ with the given symplectic form.
\begin{proof} Let us again start by noting that $X$ is a toric bundle of $\mathbb{S}P^{1}_{\mathbf{v}_1,\mathbf{w}_1}$ over $\mathbb{S}P^{k}_{\mathbf{v}_2,\mathbf{w}_2}$ or vice versa. Moreover, since Proposition \ref{coho} is true for $k_1=1$ we can use the argument in the proof of Theorem \ref{NoSoc} verbatim to treat the case when $\mathbb{S}P^{1}_{\mathbf{v}_1,\mathbf{w}_1}$ is a fiber.
\newline\indent In this case the vertical interval containing $x_0$ corresponds to the suborbifold $S=\mathbb{S}P^{1}_{\mathbf{v}_1,\mathbf{w}_1}\times \{x_0\}$. It is clear that $S$ satisfies the assumptions of Lemma \ref{foot} and hence the proof is concluded.
\end{proof}
\end{tw} \begin{tw}Let $X$ be a toric orbifold symplectomorphic to $\mathbb{S}P^{1}_{\mathbf{v}_1,\mathbf{w}_1}\times\mathbb{S}P^{1}_{\mathbf{v}_2,\mathbf{w}_2}$ endowed with the symplectic form $\lambda_1\omega_{\mathbf{v}_1,\mathbf{w}_1}+\lambda_2\omega_{\mathbf{v}_2,\mathbf{w}_2}$. Furthermore, assume that both $\mathbb{S}P^{1}_{\mathbf{v}_1,\mathbf{w}_1}$ and $\mathbb{S}P^{1}_{\mathbf{v}_2,\mathbf{w}_2}$ have singular points $x_i\in\mathbb{S}P^{1}_{\mathbf{v}_i,\mathbf{w}_i}$ such that there exists a neighbourhood $U_i$ of $x_i$ such that the ranks of isotropy groups of points in $U_i\backslash\{x_i\}$ are strictly smaller than the rank of the isotropy group of $x_i$. Then $X$ is isomorphic to the standard toric structure on $\mathbb{S}P^{1}_{\mathbf{v}_1,\mathbf{w}_1}\times\mathbb{S}P^{1}_{\mathbf{v}_2,\mathbf{w}_2}$ with the given symplectic form.
\begin{proof}
Since the situation is now symmetric we can verbatim use the proof of Theorem \ref{sing} were we can assume without loss of generality that $\mathbb{S}P^{1}_{\mathbf{v}_2,\mathbf{w}_2}$ is the fiber of the bundle (and consequently we substitute $x_2$ for $x_0$).
\end{proof}
\end{tw}
\begin{rem} It is interesting to note that the previous Theorem reduces the usefulness of Theorem \ref{11} to the case when at least one of the labeled projective spaces is simply $\mathbb{S}^2\cong\mathbb{C}P^1$.
\end{rem}

\end{document}